\documentclass[11pt,letterpaper]{amsart} 
\usepackage{amsmath,amssymb,amsthm} 
\usepackage[all]{xy}
\usepackage[OT2,T1]{fontenc}
\usepackage[utf8]{inputenc}
\usepackage{enumitem}
\usepackage{graphicx}
\usepackage{hyperref}
\usepackage{tikz}
\usetikzlibrary{cd}
\usepackage{tikz-cd}
\usepackage{xcolor}
\usepackage{ stmaryrd }
\usepackage{mathrsfs}
\usepackage{mathtools}
\usepackage[english]{babel}
\usepackage{setspace}
\usepackage[normalem]{ulem}
\usepackage{mathdots}
\usepackage{microtype}

\usepackage{arydshln}

\usetikzlibrary{arrows,positioning,decorations.pathmorphing,
	decorations.markings
}
\tikzset{inner sep=0pt,
	root/.style={circle,draw,minimum size=7pt,thick},
	fatroot/.style={circle,draw,minimum size=10pt,thick},
	short root/.style={circle,fill,minimum size=7pt},
	doublearrow/.style={postaction={decorate},
		decoration={markings,mark=at position .7
			with {\arrow{angle 60}}},double distance=3pt,thick}
}

\DeclareUnicodeCharacter{00A0}{ }
\newtheorem{introthm}{Theorem}

\newtheorem{introconjecture}[introthm]{Conjecture}

\newtheorem{proposition}{Proposition}[section]
\newtheorem{definition}[proposition]{Definition}

\newtheorem{theorem}[proposition]{Theorem}
\newtheorem{lemma}[proposition]{Lemma}
\newtheorem{corollary}[proposition]{Corollary}

\numberwithin{equation}{section}

\DeclareMathOperator{\GL}{GL}
\DeclareMathOperator{\PGL}{PGL}

\DeclareMathOperator{\Lie}{Lie}

\DeclareMathOperator{\Hom}{Hom}

\DeclareMathOperator{\Gal}{Gal}

\DeclareMathOperator{\End}{End}

\DeclareMathOperator{\Iw}{Iw}

\DeclareMathOperator{\Spec}{Spec}

\DeclareMathOperator{\CNL}{CNL}

\DeclareMathOperator{\Art}{Art}
\DeclareMathOperator{\val}{val}

\newcommand{\cO}{\mathcal{O}}

\newcommand{\Frob}{\operatorname{Frob}}

\DeclareSymbolFont{cyrletters}{OT2}{wncyr}{m}{n}
\DeclareMathSymbol{\Sha}{\mathalpha}{cyrletters}{"58}

\DeclareMathOperator{\Res}{Res}

\DeclareMathOperator{\disc}{disc}
\DeclareMathOperator{\ord}{ord}

\title{Towards the Fontaine--Mazur conjecture for $\GL_2$}

\author{Jack A. Thorne}

\begin{document}

\maketitle

\begin{abstract}
    We combine a new type of modularity result with the geometry of numbers in order to prove some new cases of the Fontaine--Mazur conjecture for $\GL_{2}$.
\end{abstract}

\tableofcontents

\section{Introduction}

The main goal of this paper is to prove some new cases of the following well-known conjecture, a special case of conjectures stated in \cite{Fon95}:
\begin{introconjecture}\label{introconj_FM_general}
    Let $p$ be a prime, let $F$ be a totally real number field, and let $\rho : G_F \to \GL_2(\overline{\mathbb{Q}}_p)$ be a continuous, irreducible representation satisfying the following conditions:
     \begin{enumerate}
        \item $\rho$ is unramified at all but finitely many places of $F$. 
        \item For each $p$-adic place of $F$, $\rho|_{G_{F_v}}$ is de Rham, of distinct Hodge--Tate weights. 
        \item $\det \rho$ is totally odd. 
    \end{enumerate}
    Then $\rho$ arises from geometry, in the sense that there is an algebraic variety $X$ over $F$ such that $\rho$ is isomorphic to a subquotient of a  Tate twist of the \'etale cohomology of $X$ (with $\overline{\mathbb{Q}}_p$-coefficients). 
\end{introconjecture}
(See \S \ref{subsec_notation} below for any undefined notation.) Many cases of this conjecture are known, proved by combining modularity lifting theorems with the technique of potential modularity introduced in \cite{Tay02}. However, these results generally only apply in the case where the (semisimple) residual representation $\overline{\rho} : G_F \to \GL_2(\overline{\mathbb{F}}_p)$ is sufficiently non-degenerate (e.g.\ irreducible -- see \cite{Bar14} for general theorems of this type). This is due to the difficulty of proving modularity lifting theorems without this hypothesis. The paper \cite{Ski99} proves modularity theorems even in the residually reducible case, but gives generally applicable results only in the case $F = \mathbb{Q}$, via a method that has so far resisted generalisation to the case of a general totally real field. 

In this paper, we give a new approach to Conjecture \ref{introconj_FM_general} that applies equally well in the case where the residual representation $\overline{\rho}$ is reducible. Our main result is as follows. 
\begin{introthm}\label{introthm_potmod}
    Let $F$ be a totally real number field, let $p$ be a prime, and let $\rho : G_F \to \GL_2(\overline{\mathbb{Q}}_p)$ be a continuous, irreducible representation satisfying the following conditions:
    \begin{enumerate}
        \item $\rho$ is unramified at all but finitely many places. 
        \item For each place $v | p $ of $F$, $\rho|_{G_{F_v}}$ is potentially crystalline and ordinary of Hodge--Tate weights $\{0, 1 \}$. 
        \item $\det \rho$ is totally odd. 
    \end{enumerate}
     Then $\rho$ is potentially modular, in the sense that there is a finite totally real extension $F' / F$, an isomorphism $\iota : \overline{\mathbb{Q}}_p \to \mathbb{C}$, and a cuspidal, regular algebraic automorphic representation $\pi$ of $\GL_2(\mathbb{A}_{F'})$ such that $\rho|_{G_{F'}} \cong r_{\pi, \iota}$. 
\end{introthm}
We recall that a potentially modular Galois representation necessarily lives in a compatible system, even over the base field \cite[Theorem 1.1]{Die04}. In the context of Theorem \ref{introthm_potmod}, $\rho$ will provably arise from geometry provided e.g.\ that there is a finite place of $F$ such that $\mathrm{WD}(\rho|_{G_{F_v}})$ is not potentially unramified. (This ensures the existence of a Jacquet--Langlands transfer of the automorphic representation potentially associated to $\rho$ to a Shimura curve, in which case one can appeal to \cite{Car86}.) 

Theorem \ref{introthm_potmod} applies in particular when the semisimple residual representation $\overline{\rho}$ is reducible. (One can generate examples of representations $\rho$ of this type using \cite[Theorem 5.2]{Fak22}.) As such, it has applications to the following conjecture in the case $F = \mathbb{Q}$, also extracted from \cite{Fon95}: 
\begin{introconjecture}\label{introconj_FM}
    Let $p$ be a prime, and let $\rho : G_\mathbb{Q} \to \GL_2(\overline{\mathbb{Q}}_p)$ be a continuous, irreducible representation satisfying the following conditions:
    \begin{enumerate}
        \item $\rho$ is unramified at all but finitely many primes.
        \item $\rho|_{G_{\mathbb{Q}_p}}$ is de Rham of distinct Hodge--Tate weights.
        \item $\det \rho$ is odd. 
    \end{enumerate}
    Then $\rho$ is modular, in the sense that there is a cuspidal, regular algebraic automorphic representation $\pi$ of $\GL_2(\mathbb{A}_\mathbb{Q})$ and an isomorphism $\iota : \overline{\mathbb{Q}}_p \to \mathbb{C}$ such that $\rho \cong r_{\pi, \iota}$. 
\end{introconjecture}
At this point in time, the vast majority of cases of this conjecture are known, thanks to the combination of (1) Serre's conjecture, asserting the modularity of the residual representation $\overline{\rho} : G_\mathbb{Q} \to \GL_2(\overline{\mathbb{F}}_p)$ \cite{Kha09, Kha09a}; and (2) powerful modularity lifting theorems, often depending on the $p$-adic local Langlands correspondence for $\GL_2(\mathbb{Q}_p)$ (see e.g.\ \cite{Col14, Kis09a, Eme11}). In particular, Conjecture \ref{introconj_FM} is known when $p \geq 3$ \cite{Pan25, Zha25}, and when $p = 2$, provided that $\overline{\rho}$ is irreducible, with non-solvable image \cite{Tun21}. (See also \cite{All14} for results in the solvable but irreducible case.) We prove the following result, which applies in particular in the open case $p = 2$, $\overline{\rho}$ reducible:
\begin{introthm}\label{introthm_FM}
     Let $p$ be a prime, and let $\rho : G_\mathbb{Q} \to \GL_2(\overline{\mathbb{Q}}_p)$ be a continuous, irreducible representation satisfying the following conditions:
     \begin{enumerate}
         \item $\rho$ is unramified at all but finitely many primes.
         \item $\rho|_{G_{\mathbb{Q}_p}}$ is potentially crystalline and ordinary of Hodge--Tate weights $\{ 0, 1 \}$.
         \item $\det \rho$ is odd.
     \end{enumerate}
     Then $\rho$ is modular. 
\end{introthm}
    (We deduce this from Theorem \ref{introthm_potmod}, together with Serre's conjecture: the theorem implies that $\rho$ lives in a weight 0, rank 2 compatible system, and the modularity of such compatible systems is a well-known consequence of Serre's conjecture \cite{Kha10}.)

    We now describe our approach to proving Theorem \ref{introthm_potmod}. The first main tool is a new kind of modularity theorem, Theorem \ref{thm_modularity_by_close_approximation}. This asserts that if $\rho, \rho' : G_F \to \GL_2(\cO)$ are continuous representations (with coefficients in a $p$-adic DVR $\cO$, of uniformizer $\varpi$), $\rho$ is modular, and $\rho$ is `close enough' to $\rho'$, then $\rho'$ is also modular. Here `close enough' is taken to mean that there is an isomorphism $\rho \text{ mod }\varpi^C \cong \rho' \text{ mod }\varpi^C$ for some integer $C = C(\rho) \geq 1$. 
    
    The classical case of a modularity lifting theorem is when $C = 1$; the important point in our theorem is that $C$ can be taken only to depend on local data (e.g. the representations $\rho|_{G_{F_v}}$, for ramified places $v$) and abstract global data (e.g. the image subgroup $\rho(G_F) \leq \GL_2(\cO)$), but not on subtle global invariants such as the rank or cardinality of associated Selmer groups. In particular, one can always find, given $\rho$, a finite set $T$ of finite places of $F$ such that if $F' / F$ is any $T$-split totally real extension, then $C(\rho) = C(\rho|_{G_{F'}})$.

    The next main tool constructs, for any given $\rho : G_F \to \GL_2(\cO)$, suitable modular approximations $\rho' : G_{F'} \to \GL_2(\cO)$ (i.e.\ modular Galois representations such that $\rho|_{G_{F'}} \text{ mod }\varpi^C \cong \rho' \text{ mod }\varpi^C$). Again, the case $C = 1$ is well-studied (see e.g.\ \cite{Tay02, Tay06}), but few results of this type exist for general $C$ (although see \cite{Gur10}). For our application, we must construct abelian varieties over completions $F_v$ whose Tate modules approximate the given Galois representation $\rho|_{G_{F_v}}$, and the geometry of numbers makes an appearance here. 

    With these two ingredients in hand, Theorem \ref{introthm_potmod} follows easily.
   
    \subsection*{Acknowledgements}

    The author's work is funded by the European Union (ERC CoG-101169866). Views and opinions expressed are however those of the author(s) only and do not necessarily reflect those of the European Union or the European Research Council. Neither the European Union nor the granting authority can be held responsible for them.

    \subsection{Notation}\label{subsec_notation}

A base number field $F$ having been fixed, we will also choose algebraic closures $\overline{F}$ of $F$ and $\overline{F}_v$ of $F_v$ for every finite place $v$ of $F$. If $p$ is a prime, then we will also write $S_p$ for the set of places of $F$ above $p$, $\overline{\mathbb{Q}}_p$ for a fixed choice of algebraic closure of $\mathbb{Q}_p$, and $\val_p$ for the $p$-adic valuation on $\overline{\mathbb{Q}}_p$ normalized so that $\val_p(p) = 1$. These choices define the absolute Galois groups $G_F = \Gal(\overline{F}/F)$ and $G_{F_v} = \Gal(\overline{F}_v/F_v)$.  We write $I_{F_v} \subset G_{F_v}$ for the inertia subgroup. We also fix embeddings $\overline{F} \hookrightarrow \overline{F}_v$, extending the canonical embeddings $F \hookrightarrow F_v$. This determines for each place $v$ of $F$ an embedding $G_{F_v} \to G_{F}$. We write $\mathbb{A}_F$ for the adele ring of $F$, and $\mathbb{A}_F^\infty = \prod'_{v \nmid \infty} F_v$ for its finite part. If $v$ is a finite place of $F$, then we write $k(v)$ for the residue field at $v$ and $q_v = | k(v) |$. If we need to fix a choice of uniformizer of $\cO_{F_v}$, then we will denote it $\varpi_v$. 

If $S$ is a finite set of finite places of $F$, then we write $F_S$ for the maximal subfield of $\overline{F}$ unramified outside $S$ and the archimedean places, and $G_{F, S} = \Gal(F_S/F)$; this group is naturally a quotient of $G_F$. If $v \not\in S$ is a finite place of $F$, then the map $G_{F_v} \to G_{F, S}$ factors through the unramified quotient of $G_{F_v}$, and we write $\Frob_v \in G_{F, S}$ for the image of a \emph{geometric} Frobenius element. We write $\epsilon : G_F \to \mathbb{Z}_p^\times$ for the $p$-adic cyclotomic character; if $v$ is a finite place of $F$, not dividing $p$, then $\epsilon(\Frob_v) = q_v^{-1}$. 

If $\rho : G_F \to \GL_n(\overline{\mathbb{Q}}_p)$ is a continuous representation, we say that $\rho$ is de Rham if for each place $v | p$ of $F$, $\rho|_{G_{F_v}}$ is de Rham. In this case, we can associate to each embedding $\tau : F \hookrightarrow \overline{\mathbb{Q}}_p$ a multiset $\mathrm{HT}_\tau(\rho)$ of Hodge--Tate weights, which depends only on $\rho|_{G_{F_v}}$, where $v$ is the place of $F$ induced by $\tau$. This multiset has $n$ elements, counted with multiplicity. There are two natural normalizations for $\mathrm{HT}_\tau(\rho)$ which differ by a sign, and we choose the one with $\mathrm{HT}_\tau(\epsilon) = \{ -1 \}$ for every choice of $\tau$. A representation $\rho : G_{F_v} \to \GL_2(\overline{\mathbb{Q}}_p)$ is ordinary with Hodge--Tate weights $\{ 0, 1 \}$ if there is an isomorphism
\[ \rho \sim \left( \begin{array}{cc} \psi_1 & \ast \\ 0 & \epsilon^{-1} \psi_2 \end{array} \right), \]
where $\psi_1, \psi_2 : G_{F_v} \to \overline{\mathbb{Q}}_p^\times$ are finitely ramified characters (any such representation is de Rham, with $\mathrm{HT}_\tau(\rho) = \{ 0, 1 \}$ for any $\tau$). 

If $v$ is a finite place of $F$ and $\rho : G_{F_v} \to \GL_n(\overline{\mathbb{Q}}_p)$ is a continuous representation, which is de Rham if $v | p$, then we write $\mathrm{WD}(\rho)$ for the associated Weil--Deligne representation, which is uniquely determined, up to isomorphism.

We will call a finite extension $E/\mathbb{Q}_p$ inside $\overline{\mathbb{Q}}_p$ a coefficient field. A coefficient field $E$ having been fixed, we will write $\cO$ for its ring of integers, $k$ for its residue field, and $\varpi$ for a fixed choice of uniformizer. If $M$ is an $\cO$-module, then we write $M^\vee = \Hom_\cO(M, E / \cO)$. 

If $A$ is a complete Noetherian local $\cO$-algebra with residue field $k$, then we write $\mathfrak{m}_A \subset A$ for its maximal ideal, and $\CNL_A$ for the category of complete Noetherian local $A$-algebras with residue field $k$. We endow each object $R \in \CNL_A$ with its profinite ($\mathfrak{m}_R$-adic) topology. 

\section{Commutative algebra}\label{sec_commuative_algebra}

In this section, we prove some preparatory results that, although elementary, play a decisive role in the proof of our main modularity result, Theorem \ref{thm_modularity_by_close_approximation}. 

First, in \S \ref{subsec_finite_abelian_groups}, we prove some results about finite abelian groups, that make precise the notion of being `close to' a finite free $\mathbb{Z} / p^N \mathbb{Z}$-module; these will be useful in analyzing Selmer groups with torsion coefficients. 

Second, in \S \ref{subsec_hensels_lemma}, we give some consequences of Hensel's lemma, proving in particular the key Theorem \ref{thm_application_of_Hensel}. This states that, if $R$ is a complete Noetherian local $\mathbb{Z}_p$-algebra, and $f, f' : R \to \mathbb{Z}_p$ are homomorphisms that are sufficiently close in the $p$-adic topology on $(\Spec R)(\mathbb{Z}_p)$, then they define points on the same irreducible component of $\Spec R$. This result enters during the usual approach to proving modularity by analyzing the support of a suitable space of modular forms, viewed as a closed subspace of the spectrum of a Galois (pseudo-)deformation ring. 

\subsection{Finite abelian groups}\label{subsec_finite_abelian_groups}

Let $E / \mathbb{Q}_p$ be a coefficient field.
\begin{lemma}\label{lem_submodule_of_finite_module}
    Let $A, B$ be finite $\cO / (\varpi^N)$-modules. Then the following are equivalent:
    \begin{enumerate}
        \item There exists an injective homomorphism $B \to A$.
        \item There exists a surjective homomorphism $A \to B$.
        \item Writing $A \cong \oplus_{i=1}^r \cO / (\varpi^{d_i})$ and $B \cong \oplus_{i=1}^s \cO / (\varpi^{e_i})$ with $d_1 \geq d_2 \geq \dots \geq d_r \geq 1$ and $e_1 \geq e_2 \geq \dots \geq e_s \geq 1$, we have $r \geq s$ and $d_i \geq e_i$ for each $1 \leq i \leq s$. 
    \end{enumerate}
\end{lemma}
\begin{proof}
    (1) and (2) are equivalent by duality. It is clear that (3) implies (1). We show that (1) implies (3). We have $r = \dim_k A[\varpi]$ and $s = \dim_k B[\varpi]$, so if $B \leq A$ then $B[\varpi] \leq A[\varpi]$ and hence $s \leq r$. 
    
    On the other hand, for any $t \geq 1$, $\dim_k A[\varpi^{t}] / A[\varpi^{t-1}]$ is equal to the number of $i$ such that $d_i \geq t$ (and similarly for $B$). Suppose for contradiction that $e_i > d_i$ for some $i$, and let $i$ be minimal with this property. Then $e_1 \leq d_1, \dots e_{i-1} \leq d_{i-1}$, but $e_i > d_i$. Taking $t  = e_i$, we see that we have
    \[ \dim_k A[\varpi^{t}] / A[\varpi^{t-1}] = i-1, \]
    \[ \dim_k B[\varpi^{t}] / B[\varpi^{t-1}] \geq i, \]
    a contradiction, since there is an embedding
    \[ B[\varpi^{t}] / B[\varpi^{t-1}] \to A[\varpi^{t}] / A[\varpi^{t-1}]. \]
\end{proof}
\begin{lemma}\label{lem_c-free_module}
    Let $A$ be a finite $\cO / (\varpi^N)$-module. Suppose given integers $a, c, N \geq 0$ such that $N > 2 c$. Then the following are equivalent:
    \begin{enumerate}
        \item There exists a homomorphism $(\cO / (\varpi^N))^a \to A$ with kernel and cokernel annihilated by $\varpi^c$. 
        \item There exists an isomorphism $A \cong \oplus_{i=1}^a \cO / (\varpi^{N - x_i}) \oplus T$, where $T$ is an $\cO$-module such that $\varpi^c T = 0$, and $0 \leq x_i \leq c$.
    \end{enumerate}
    In either of these cases, $a$ is uniquely determined by $A$. 
\end{lemma}
\begin{proof}
    It is immediate that (2) implies (1). To see that (1) implies (2), suppose given a homomorphism $f : (\cO / (\varpi^N))^a \to A$ with the claimed property, and let $I$ denote the image of $f$, $Q = A / I$. Write $A \cong \oplus_{i=1}^r \cO / (\varpi^{d_i})$ with $d_1 \geq d_2 \geq \dots \geq d_r \geq 1$. Applying Lemma \ref{lem_submodule_of_finite_module} to $I \leq A$, we see that we have $a \leq r$ and $d_1 \geq d_2 \geq \dots d_a \geq N - c$. Since $\varpi^c Q = 0$, we have $\varpi^c A \leq I$. Applying again Lemma \ref{lem_submodule_of_finite_module}, we find that $c \geq d_{a+1} \geq \dots \geq d_r$. Thus $A$ has the claimed form.

    It remains to show that $a$ is uniquely determined by $A$. This follows from the characterization in (2) provided that $N-c > c$, i.e. that $N > 2c$. This completes the proof. 
\end{proof}
\begin{definition}
    Let $A$ be a finite $\cO / (\varpi^N)$-module. Suppose given integers $a, c, N \geq 0$ such that $N > 2 c$. We say that $A$ is a $c$-free $\cO / (\varpi^N)$-module of rank $a$ if it satisfies the equivalent conditions of Lemma \ref{lem_c-free_module}.
\end{definition}
If $A$ is a $c$-free $\cO / (\varpi^N)$-module, and we're given a homomorphism $A \to B$ with kernel and cokernel killed by $\varpi^d$, then $B$ is $c+d$-free of the same rank (provided that $N > 2(c+d)$). The following proposition will also be useful to prove $c$-freeness in  certain cases: 
\begin{proposition}\label{prop_euler_characteristic}
    Let $n \geq 1$, let $N > 2^n c$, and suppose given an exact sequence 
    \[ 0 \to A_0 \to A_1 \to \dots \to A_n \to 0 \]
    of $\cO / (\varpi^N)$-modules, where each $A_i$ ($i = 1, \dots, n$) is $c$-free of rank $a_i \geq 0$. Then $A_0$ is $2^{n-1} c$-free of rank $a_0 = \sum_{i=1}^n (-1)^{i-1} a_i$.
\end{proposition}
\begin{proof}
    The proof is by induction. If $n = 1$, then $A_0 = A_1$ and there is nothing to prove. If $n = 2$, then we have a short exact sequence
    \[ 0 \to A_0 \to A_1 \to A_2 \to 0. \]
    Let $\phi : A_1 \to A_2$ denote the map appearing in this sequence, and let $\overline{\phi} : A_1 / (\varpi^{N-c}) \to A_2 / (\varpi^{N-c})$ be the map arising by passage to quotient. Then we have short exact sequences
    \[ 0 \to \ker \overline{\phi} \to A_1 / (\varpi^{N-c}) \to A_2 / (\varpi^{N-c}) \to 0 \]
    and (by the snake lemma)
    \[ A_2[\varpi^{N-c}] / \phi A_1[\varpi^{N-c}] \to A_0 / (\varpi^{N-c}) \to \ker \overline{\phi} \to 0. \]
    The first of these shows that $a_1 \geq a_2$ and that there is an isomorphism $\ker \overline{\phi} \cong \cO / (\varpi^{N-c})^{a_1 - a_2} \oplus T$, where $\varpi^c T = 0$. The group $A_2[\varpi^{N-c}] / \phi A_1[\varpi^{N-c}]$ is annihilated by $\varpi^c$, so the second short exact sequence shows that there is an isomorphism $A_0 / (\varpi^{N-c}) \cong (\cO / (\varpi^{N-c}))^{a_1 - a_2} \oplus T'$, where $\varpi^{2c} T' = 0$. It follows that $A_0$ is $2c$-free of rank $a_1 - a_2$. This completes the proof of the desired statement in the case $n = 2$.

    For the induction step, suppose that the result is known for $1 \leq n \leq r$, with $r \geq 2$, and consider an exact sequence
    \[ 0 \to A_0 \to A_1 \to \dots \to A_r \to A_{r+1} \to 0, \]
    where each $A_i$ ($i = 1, \dots, r+1$) is $c$-free of rank $a_i$. Consider the two exact sequences
    \[ 0 \to K_r \to A_r \to A_{r+1} \to 0 \]
    and
    \[ 0 \to A_0 \to A_1 \to \dots \to K_r \to 0. \]
    Applying the proposition in the case $r = 2$, we see that $K_r$ is $2c$-free of rank $a_r - a_{r+1}$. Applying the proposition in the case $n = r$, we deduce that $A_0$ is $2^{r-1} \cdot 2c = 2^r c$-free of rank $\sum_{i=1}^{r+1} (-1)^{i-1} a_i$, as required. 
\end{proof}
Here's a slight variation that will also be useful. 
\begin{lemma}\label{lem_existence_of_homomorphism}
    Let $M$ be a finite $\cO$-module, and let $q, C \geq 0$. Then the following are equivalent:
    \begin{enumerate}
        \item There exists a morphism $\cO^q \to M$ with cokernel annihilated by $\varpi^C$.
        \item There exists a morphism $\cO^q \to M / (\varpi^{C+1})$ with cokernel annihilated by $\varpi^C$. 
    \end{enumerate}
    If these conditions hold and  $\dim_E M \otimes_\cO E = q$, then there is an isomorphism $M \cong \cO^q \oplus T$, where $T$ is a finite $\cO$-module that is annihilated by $\varpi^C$. 
\end{lemma}
\begin{proof}
    It is clear that (1) implies (2). Suppose that (2) holds. Then we can lift our morphism to a morphism $\cO^q \to M$. Let $N$ denote its image. Our hypothesis states that $\varpi^C M \leq N + \varpi^{C+1} M$. We must show that $\varpi^C M \leq N$. This follows on applying Nakayama's lemma to the quotient module $(\varpi^C M + N) / N$.

    If (1) holds, and $\dim_E M \otimes_\cO E = q$, then we need to show that $\varpi^C$ annihilates the torsion submodule of $M$. The morphism $\cO^q \to M$ must be injective; let $N$ be its image. If $x \in M$ is a torsion element, then $\varpi^C x \in N$, hence $\varpi^C x = 0$ (as $N$ is torsion-free). This concludes the proof. 
\end{proof}

\subsection{Some consequences of Hensel's Lemma}\label{subsec_hensels_lemma}

The approach we take here extends \cite[Ch. III, \S 4]{Bou89} and \cite{Fis97}. 

We fix a coefficient field $E / \mathbb{Q}_p$. If $f_1, \dots, f_r \in \cO \llbracket X_1, \dots, X_n\rrbracket$, and $\underline{a} = (a_1, \dots, a_n) \in (\varpi)^n$, then we write $M_f(\underline{a}) = (\partial_j f_i(\underline{a})) \in M_{r \times n}(\cO)$ for the matrix of partial derivatives. 
\begin{lemma}\label{lem_inverting_power_series}
    Let $n \geq 1$, and let $f_1, \dots, f_n \in (X_1, \dots, X_n) \cO \llbracket X_1, \dots, X_n \rrbracket$. Suppose that $\det M_f(0) \in \cO^\times$. Then there exist 
    \[ g_1, \dots, g_n \in (X_1, \dots, X_n) \cO \llbracket X_1, \dots, X_n \rrbracket \]
    such that $\underline{f}(\underline{g}(\underline{X})) = (X_1, \dots, X_n)$. The tuple $\underline{g}$ is unique and satisfies $\underline{g}(\underline{f}) = (X_1, \dots, X_n)$.
\end{lemma}
\begin{proof}
    See \cite[Ch. III, \S 4.4, Proposition 5]{Bou89}. 
\end{proof}
\begin{corollary}\label{cor_power_series_injective}
    Let $n \geq 1$, and let  $f_1, \dots, f_n \in (X_1, \dots, X_n) \cO \llbracket X_1, \dots, X_n \rrbracket$. Suppose that $\det M_f(0) \in \cO^\times$. Then the map $(\varpi)^n \to (\varpi)^n$, $\underline{a} \mapsto \underline{f}(\underline{a})$, is bijective. 
\end{corollary}
\begin{proposition}\label{prop_inverse_function_theorem}
    Let $n \geq 1$, and let $f_1, \dots, f_n \in (X_1, \dots, X_n) \cO \llbracket X_1, \dots, X_n \rrbracket$. Suppose that there exists $C \geq 0$ and $M' \in M_{n \times n}(\cO)$ such that $M_f(0) \cdot M' = \varpi^C 1_n$. Then there exist $g_1, \dots, g_n, h_1, \dots, h_n \in (X_1, \dots, X_n) \cO \llbracket X_1, \dots, X_n \rrbracket$ satisfying the following conditions:
    \begin{enumerate}
    \item $\underline{f}(\varpi^C \underline{X}) = \varpi^C M_f(0) \underline{g}(\underline{X})$. 
    \item $M_g(0) = 1_n$ and $\underline{g}(\underline{h}) = 1_n$.
        \item $\underline{f}( \varpi^C \underline{h}(X)) = \varpi^C M_f(0) \cdot \underline{X}$.
    \end{enumerate}
\end{proposition}
\begin{proof}
We have $\underline{f}(\varpi^C \underline{X}) = \varpi^C M_f(0) \cdot \underline{X} + \varpi^{2C} \underline{r}(X)$, where each $r_i(\underline{X})$ $(i = 1,\dots, n$) lies in $(X_1, \dots, X_n)^2 \cO \llbracket X_1, \dots, X_n \rrbracket$. Applying our hypothesis, we have
\[ \underline{f}(\varpi^C \underline{X}) = \varpi^C M_f(0) \cdot \underline{X} + \varpi^{C} M_f(0) \cdot M' \cdot \underline{r}(\underline{X}) = \varpi^C M_f(0) \cdot \underline{g}(\underline{X}),   \]
where we take $\underline{g}(\underline{X}) = \underline{X} + M' \cdot \underline{r}(\underline{X})$. This gives (1). Then (2) and (3) follow from Lemma \ref{lem_inverting_power_series}. 
\end{proof}
\begin{corollary}\label{cor_Hensel_uniqueness}
    Let $n \geq 1$, and let $f_1, \dots, f_n \in (X_1, \dots, X_n) \cO \llbracket X_1, \dots, X_n \rrbracket$. Suppose that there exists $C \geq 0$ and $M' \in M_{n \times n}(\cO)$ such that $M_f(0) \cdot M' = \varpi^C 1_n$. If $\underline{a}, \underline{a}' \in (\varpi^{C+1})^n$ and $\underline{f}(\underline{a}) = \underline{f}(\underline{a}')$, then $\underline{a} = \underline{a}'$.
\end{corollary}
\begin{proof}
    Substituting into Proposition \ref{prop_inverse_function_theorem}, we obtain
    \[ \underline{f}(\underline{a}) = \varpi^C M_f(0) \underline{g}(\varpi^{-C} \underline{a}) = \varpi^C M_f(0) \underline{g}(\varpi^{-C} \underline{a}'), \]
    hence $\underline{g}(\varpi^{-C} \underline{a}) = \underline{g}(\varpi^{-C} \underline{a}')$. Corollary \ref{cor_power_series_injective} then implies that $\underline{a} = \underline{a}'$, as desired. 
\end{proof}
\begin{proposition}\label{prop_implicit_function_theorem}
    Let $q \geq r \geq 1$, and let $f_1, \dots, f_r \in (X_1, \dots, X_q) \cO \llbracket X_1, \dots, X_q \rrbracket$. Let $N \in M_{r \times r}(\cO)$ denote the matrix obtained by taking the first $r$ columns of $M_f(0)$, and suppose that there exists $C \geq 0$ and $M' \in M_{r \times r}(\cO)$ such that $N \cdot M' = \varpi^C 1_r$. Then we can find $\phi_{r+1}, \dots, \phi_{q} \in (X_{r+1}, \dots, X_q) \cO \llbracket X_{r+1}, \dots, X_q \rrbracket$ such that 
    \[ \underline{f}(\varpi^C \phi_{1}, \dots, \varpi^C \phi_r, \varpi^{2C} X_{r+1}, \dots, \varpi^{2C} X_q) = 0. \]
\end{proposition}
\begin{proof}
    We apply Proposition \ref{prop_inverse_function_theorem} to $\underline{u} = (f_1, \dots, f_r, X_{r+1}, \dots, X_{q})$. We have
    \[ M_u(0) = \left( \begin{array}{cc} N & Z \\ 0 & 1_{q-r} \end{array} \right) \]
    for some $Z \in M_{r \times (q-r)}(\cO)$, hence
    \[ M_u(0) \cdot \left( \begin{array}{cc} M' & - M' Z \\ 0 & \varpi^C 1_{q-r} \end{array} \right) = \varpi^C 1_q,  \]
    so this application is justified. Inspecting the proof of Proposition \ref{prop_inverse_function_theorem}, we see that we will have $g_i(\underline{X}) = X_i$ if $i = r+1, \dots, q$, and hence $h_i(\underline{X}) = X_i$ if $i = r+1, \dots, q$ (by uniqueness of $\underline{h}$). 

    Writing $\underline{X} = (\underline{X}^{(1)}, \underline{X}^{(2)})$, where the first tuple has $\underline{X}^{(1)} = (X_1, \dots, X_r)$ and $\underline{X}^{(2)} = (X_{r+1}, \dots, X_q)$, we see that we have
    \[ \underline{f}(\varpi^C \underline{h}(-M' Z \underline{X}^{(2)}, \varpi^C \underline{X}^{(2)})) = - \varpi^C N M' Z \underline{X}^{(2)} + \varpi^{2C} Z \underline{X}^{(2)} = 0. \]
    The proof is therefore complete on taking $\phi_i = h_i(-M' Z \underline{X}^{(2)}, \varpi^C \underline{X}^{(2)})$ for $i = 1, \dots, r$. 
\end{proof}
\begin{corollary}\label{cor_close_components}
    Let $g \geq 1$, and let $f_1, \dots, f_r \in (\varpi, X_1, \dots, X_g) \cO \llbracket X_1, \dots, X_g \rrbracket$. Let $B = \cO \llbracket X_1, \dots, X_g \rrbracket / (f_1, \dots, f_r)$, and suppose given $\cO$-algebra homomorphisms $\psi_1, \psi_2 : B \to \cO$ and an integer $C \geq 0$ satisfying the following conditions:
    \begin{enumerate}
        \item Let $P_i = \ker \psi_i$. Then there exists an isomorphism $P_1 / P_1^2 \cong \cO^{g-r} \oplus T$, where $T$ is a finite $\cO$-module such that $\varpi^C T = 0$.
        \item $\psi_1 \equiv \psi_2 \text{ mod }\varpi^{2C+1}$. 
    \end{enumerate}
    Then there is a unique minimal prime $Q$ of $B$ contained in $P_1$, and it is contained in $P_2$ also. 
\end{corollary}
Expressed in more geometric language: $P_1$ lies on a unique irreducible component of $\Spec B$, and $P_2$ lies on the same component. 
\begin{proof}
    After a change of variable, we can assume that the homomorphism $\cO \llbracket X_1, \dots, X_g \rrbracket \to \cO$ induced by $\psi_1$ is given by evaluation at $0$ (and consequently, $f_i \in (X_1, \dots, X_g) \cO \llbracket X_1, \dots, X_g \rrbracket$ for each $i = 1 ,\dots, r$). Then $\psi_2$ is given by evaluation at some point $\underline{a} \in (\varpi^{2C+1})^g$. 

    The $\cO$-module $P_1 / P_1^2$ is isomorphic to the cokernel of the map $\cO^r \to \cO^g$ induced by $M_f(0)^t$. Thanks to the existence of the Smith normal form, we can assume, after acting by $\GL_g(\cO)$ on the variables $X_1, \dots, X_g$ and by $\GL_r(\cO)$ on the equations $f_1, \dots, f_r$, that we have
    \[ M_f(0) = \left( \begin{array}{cccc|} \varpi^{d_1} & & & \\ & \varpi^{d_2} & & \\ & & \ddots & \\ & & & \varpi^{d_r} \end{array} \,\,\,\, 0_{g-r} \right), \]
    where $T \cong \oplus_{i=1}^r \cO / (\varpi^{d_i})$, hence (as $\varpi^C T = 0$) $d_i \leq C$ for each $i = 1, \dots, r$. In particular, we are in the situation of Proposition \ref{prop_implicit_function_theorem}, and so we have $\phi_1, \dots, \phi_r \in (X_{r+1}, \dots, X_g) \cO \llbracket X_{r+1}, \dots, X_g \rrbracket$ such that 
    \[ \underline{f}(\varpi^C \phi_1, \dots, \varpi^C \phi_r, \varpi^{2C} X_{r+1}, \dots, \varpi^{2C} X_g) = 0. \]
    We can define an $\cO$-algebra homomorphism $\Psi : B \to \cO \llbracket X_{r+1}, \dots, X_g \rrbracket$ by the formula $X_i \mapsto \varpi^C \phi_i$ if $1 \leq i \leq r$, $X_i \mapsto \varpi^{2C} X_i$ if $r+1 \leq i \leq g$. 

    The homomorphism $\psi_1$ factors through $\Psi$, so $\ker \Psi \leq P_1$. The image of $\Psi$ is a domain, so $\ker \Psi$ is prime. Moreover, $B_{P_1}$ is regular, so there is in fact a unique minimal prime $Q$ of $B$ contained in $P_1$, satisfying $Q \leq \ker \Psi \leq P_1$. We must show that $Q \leq P_2$; it will suffice to show that $\ker \Psi \leq P_2$, or equivalently that $\psi_2$ factors through $\Psi$. Define another homomorphism $\psi_3 : B \to \cO$ by the formula
    \[ \psi_3 = \operatorname{ev}_{\varpi^{-2C} \underline{a}^{(2)}} \circ \Psi = \operatorname{ev}_{(\varpi^C \underline{\phi}(\varpi^{-2C} \underline{a}^{(2)}), a^{(2)})}. \]
    We'll be done if we can show that $\psi_2 = \psi_3$; equivalently, if $\underline{a}^{(1)} = \varpi^C \underline{\phi}(\varpi^{-2C} \underline{a}^{(2)})$. To this end, we set $\underline{b} = \varpi^C \underline{\phi}(\varpi^{-2C} \underline{a}^{(2)}) \in (\varpi^{C+1})^r$, and $\underline{u} = \underline{f}(X_1, \dots, X_r, \underline{a}^{(2)}) - \underline{f}(0, \underline{a}^{(2)}) \in (X_1, \dots, X_r) \cO \llbracket X_1, \dots, X_r \rrbracket$. Then $M_u(0) = \operatorname{diag}(\varpi^{d_1}, \dots, \varpi^{d_r})$, so Corollary \ref{cor_Hensel_uniqueness} implies that $\underline{u}$ is injective on $(\varpi^{C+1})^r$. In particular, since we have $\underline{u}(\underline{a}^{(1)}) = \underline{u}(\underline{b})$, we conclude that $\underline{a}^{(1)} = \underline{b}$. This completes the proof. 
\end{proof}
\begin{theorem}\label{thm_application_of_Hensel}
    Let $q \geq 0$, let $R$ be a complete Noetherian local $\cO$-algebra, and let $\psi_1, \psi_2 : R \to \cO$ be $\cO$-algebra homomorphisms. Suppose given an integer $C \geq 0$ satisfying the following conditions:
    \begin{enumerate}
        \item Let $P_i = \ker \psi_i$. Then $\dim R_{P_1} \geq q$, and there is an $\cO$-algebra homomorphism $\cO^q \to P_1 / P_1^2$ with cokernel annihilated by $\varpi^C$.
        \item $\psi_1 \equiv \psi_2 \text{ mod }\varpi^{2C+1}$.
    \end{enumerate}
    Then $R_{P_1}$ is regular, there is a unique minimal prime $Q$ of $R$ contained in $P_1$, and this minimal prime is contained in $P_2$ also. 
\end{theorem}
\begin{proof}
    Fix a surjection $ A = \cO \llbracket X_1, \dots, X_g \rrbracket \to R$ of $\cO$-algebras, and let $f_1, \dots, f_s \in A$ denote a generating set for the ideal $I = \ker(A \to R)$. After a change of variable, we can assume that $\psi_1$ is given by evaluation at $\underline{X} = 0$. 

    Our hypotheses imply that $R_{P_1}$ is regular of dimension $q$. In particular, there is an isomorphism $P_1 / P_1^2 \cong \cO^q \oplus T$, where $\varpi^C T = 0$. The $\cO$-module $P_1 / P_1^2$ is isomorphic to the cokernel of the map $\cO^s \to \cO^g$ given by the matrix $M_f(0)^t$. Thanks to the existence of Smith normal form, we can assume, after acting by an element of $\GL_g(\cO)$ on $X_1, \dots, X_g$ and by an element of $\GL_s(\cO)$ on $f_1, \dots, f_s$, that we have
    \[ M_f(0) = \left( \begin{array}{cccccc} \varpi^{d_1} & & & & \\ & \varpi^{d_2} & & & \\ & & \ddots & & \\ & & & \varpi^{d_r}  & \\ & & & & 0 \\ & & & & & \ddots \end{array}  \right), \]
    where $r = g - q$ and $0 \leq d_i \leq C$ for each $i = 1, \dots, r$. Let $B = A / (f_1, \dots, f_r)$. Then the map $A \to R$ factors through $B$; let $P_1', P_2'$ denote the pullback of $P_1, P_2$, respectively, to $B$. Then $B_{P_1'}$ is regular of dimension $q$; consequently, the map $B_{P_1'} \to R_{P_1}$ is an isomorphism. Let $Q$ denote the unique minimal prime of $R$ contained in $P_1$, and let $Q'$ denote its pullback to $B$; then $Q'$ is the unique minimal prime of $B$ contained in $P_1'$. 

    The hypotheses of Corollary \ref{cor_close_components} apply to $B$, and it follows that $P_1'$ and $P_2'$ contain $Q'$. Since $B / Q' \to R / Q$ is an isomorphism, it follows that $P_1$ and $P_2$ both contain $Q$; and this is what we needed to show. 
\end{proof}

\subsection{Group determinants} 
Let $E$ be a coefficient field, and let $\Delta \leq \GL_n(\cO)$ be a closed subgroup that is absolutely irreducible in its tautological action on $E^n$. The notion of group determinant is defined in \cite{Che14}, and the relation between the deformation theory of a representation and its associated group determinant studied in \cite{New23}. We record a result  in this direction for later use. 
\begin{proposition}\label{prop_deforming_a_pseudocharacter}
    There exists an integer $C \geq 0$ with the following properties. Let $\Gamma$ be a profinite group, and let $\rho : \Gamma \to \Delta$ be a surjective, continuous homomorphism. Let $A = \cO \oplus \epsilon E / \cO$, and let $\alpha_{C} : A \to A$ be the $\cO$-algebra homomorphism that acts as multiplication by $\varpi^{C}$ on $\epsilon E / \cO$. Then:
    \begin{enumerate}
        \item Let $\widetilde{D}$ be a continuous group determinant of $\Gamma$ over $A$ lifting $D = \det(X - \rho)$. Then there exists a continuous representation $\widetilde{\rho} : \Gamma \to \GL_n(A)$ lifting $\rho$ such that $\alpha_{C} \circ \widetilde{D} = \det(X - \widetilde{\rho})$.
        \item Let $\widetilde{\rho}_1, \widetilde{\rho}_2 : \Delta \to \GL_n(A)$ be liftings of $\rho$ such that $\det(X - \widetilde{\rho}_1) = \det(X - \widetilde{\rho}_2)$. Then $\alpha_{C} \circ \widetilde{\rho}_1$ and $\alpha_{C} \circ \widetilde{\rho}_2$ are conjugate under the action of the group $1 + \epsilon M_n(E/ \cO)$; and if $1 + \epsilon X$ centralizes $\widetilde{\rho}_1$, then $\varpi^{C} X$  is a scalar matrix. 
    \end{enumerate}
\end{proposition}
\begin{proof}
    This is a restatement of \cite[Proposition 2.7]{New23} that makes slightly clearer what it means by the integer $C(\rho)$ `depending only on the image $\rho(\Gamma)$'.
\end{proof}
In the situation of the proposition, we will define $C(\Delta)$ to be the smallest value of $C \geq 0$ such that the conclusion of the proposition is satisfied. 

\section{HBAVs}\label{sec_application_of_moret-bailly}

Let $M$ be a totally real number field, and let $S$ be a scheme. In this paper, we follow \cite{Rap78} in defining an $M$-HBAV over $S$ to be pair $(X, m)$, where $X$ is a abelian scheme over $S$ and $m : \cO_M \to \End(X)$ is a homomorphism such that the $\cO_S \otimes \cO_M$-module $\Lie X$ is free, locally on $S$. If $(X, m)$ is an $M$-HBAV, then we define
\[ \mathcal{P}(X, m) = \{ \lambda \in \Hom_{\cO_M}(X, X^\vee) \mid \lambda = \lambda^\vee \}. \]
\'Etale locally on $S$, $\mathcal{P}(X, m)$ is a free $\cO_M$-module of rank 1, and has the structure of an ordered invertible $\cO_M$-module: in other words, for each embedding $M \to \mathbb{R}$, the specification of a connected component of
\[ \mathcal{P}(X, m) \otimes_{\cO_M} \mathbb{R} - \{ 0 \}, \]
namely the one containing some (equivalently, all) polarizations of $X$ \cite[Proposition 1.17]{Rap78}. If $\mathfrak{a} \leq M$ is a fractional $\cO_M$-ideal, then we give it the ordering determined by the natural one on $M$. We call a polarised $M$-HBAV over $S$ a triple $(X, m, j)$, where $(X, m)$ is an $M$-HBAV and $j : \mathfrak{d}_{M / \mathbb{Q}}^{-1} \to \mathcal{P}(X, m)$ is an isomorphism of \'etale sheaves of ordered $\cO_M$-modules. 
\begin{lemma}\label{lem_scalar_extension_and_polarisations}
    Let $M' / M$ be an extension of totally real number fields, and let $(X, m, j)$ be a polarised $M$-HBAV over a field $k$. Let $X' = X \otimes_{{\cO_M}} \cO_{M'}$, let $m' : \cO_{M'} \to \End(X')$ be the scalar extension of $m$, and let $j' :  \mathfrak{d}_{M' / \mathbb{Q}}^{-1} = \mathfrak{d}_{M / \mathbb{Q}}^{-1} \otimes_{\cO_M} \mathfrak{d}_{M' / M}^{-1} \to \mathcal{P}(X', m')$ be the scalar extension of $j$. Then $(X', m', j')$ is a polarised $M'$-HBAV.
\end{lemma}
\begin{proof}
    We have $(X')^\vee = X^\vee \otimes_{\cO_M} \Hom_{\cO_M}(\cO_{M'}, \cO_M) \cong X^\vee \otimes_{\cO_M} \mathfrak{d}_{M' / M}^{-1}$. This shows that $j'$ is defined. We need to check that it is an isomorphism, and that it preserves positivity. It is an isomorphism because we have
    \[ \Hom_{\cO_{M'}}( X \otimes_{\cO_M} \cO_{M'}, (X \otimes_{\cO_M} \cO_{M'})^\vee) \cong \Hom_{\cO_M}(X, X^\vee) \otimes_{\cO_M} \mathfrak{d}_{M' / M}^{-1}. \]
    It preserves positivity because, if $\lambda : X \to X^\vee$ is a polarisation, then $\lambda \otimes 1 \in \Hom_{\cO_{M'}}( X \otimes_{\cO_M} \cO_{M'}, X^\vee \otimes_{\cO_M} \mathfrak{d}_{M'  / M}^{-1})$ is a polarisation.
\end{proof}
The goal of this section is to prove the following result.
\begin{theorem}\label{thm_pot_mod}
    Let $F$ be a totally real number field, let $p$ be a prime, and let $E$ be a coefficient field. Suppose given a continuous, irreducible representation $\rho : G_F \to \GL_2(\cO)$ satisfying the following conditions:
    \begin{enumerate}
        \item $\det \rho = \epsilon^{-1}$, and $\rho$ is unramified at all but finitely many places of $F$.
        \item For each place $v | p$ of $F$, $\rho|_{G_{F_v}}$ is crystalline ordinary, of Hodge--Tate weights $\{ 0, 1 \}$.
        \item For every finite place $v$ of $F$, $\rho|_{G_{F_v}}$ is semistable. 
    \end{enumerate}
    Let $\Sigma$ denote the set of places $v$ of $F$ such that $\mathrm{WD}(\rho|_{G_{F_v}})$ is ramified. Then, for any $n \geq 1$, we can find the following data:
    \begin{enumerate}
        \item A totally real number field $F' / F$ such that $(\rho \times \epsilon)(G_{F'}) = (\rho \times \epsilon)(G_F)$, and each place $v$ such that $\rho|_{G_{F_v}}$ is ramified splits in $F'$. (In other words, such that each place $v \in \Sigma$ or $v | p$ splits in $F'$.)
        \item A totally real number field $M$, together with an $M$-HBAV $(A, m, j)$ over $F'$ that is modular; that has good reduction away from $\Sigma$; that has multiplicative reduction at each place of $F'$ lying above $\Sigma$; and that has ordinary reduction at each place $v | p$.
        \item A prime ideal $\mathfrak{p} \leq \cO_M$ and an embedding $E \to M_\mathfrak{p}$ making $M_\mathfrak{p}$ a compositum of quadratic extensions of $E$ such that, for each $v | p$, the unit roots of the linearised crystalline Frobenius on $\rho$ and $T_{\mathfrak{p}} A$ have the same image in the residue field of $M_\mathfrak{p}$.
        \item An isomorphism $\rho|_{G_{F'}} \otimes_\cO \cO_{M_\mathfrak{p}} / (\varpi^n) \cong (T_{\mathfrak{p}} A)/ (\varpi^n)^\vee $ of $\cO_{M_\mathfrak{p}}[G_{F'}]$-modules. 
    \end{enumerate}
\end{theorem}
To prove Theorem \ref{thm_pot_mod}, we will follow the template of \cite{Tay02} in constructing a suitable moduli space of HBAV, and then appealing to the results of \cite{Mor89}. We will in fact use the following extension of the main result of \emph{op. cit.}, given in \cite[\S 3]{Mor90}:
\begin{theorem}\label{thm_moret-bailly}
    Let $K$ be a number field, and let $S$ be a finite set of places of $K$, containing the archimedean ones. Let $f : X \to \Spec \cO_{K, S}$ be a flat, surjective morphism of finite type, with $X_K$ geometrically integral.

    Let $\Sigma \subset S$ be a subset, and suppose given for each $v \in \Sigma$ a non-empty subset $\Omega_v \subset X(K_v)$, contained in the smooth locus of $X_{K_v}$. Finally, suppose that one of the following two conditions is satisfied:
    \begin{enumerate}
        \item $\Sigma \neq S$. 
        \item $\Sigma = S$, and there exists an open immersion $j : X_K \to \overline{X}_K$, where $\overline{X}_K$ is a projective $K$-scheme, and $\overline{X}_K - X_K$ has codimension $\geq 2$ in $\overline{X}_K$. 
    \end{enumerate}
    Then we can find a finite extension $L / K$ and a point $P \in X(\cO_{L, S})$ with the following property: for all $v \in \Sigma$, $v$ splits in $L$; and for each place $w | v$ of $L$, $P$ lies in $\Omega_v \subset X(L_w)$. 
\end{theorem}
Our first task is to construct HBAVs over the completions $F_v$, that will witness the existence of local points on our moduli spaces. We begin with a consideration of Weil numbers.
\begin{lemma}\label{lem_honda_tate}
    Let $q$ be a prime power, and let $f(x) \in \mathbb{Z}[x]$ be a monic polynomial of degree $2g$, such that every complex root has absolute value $q^{1/2}$. Then the following are equivalent:
    \begin{enumerate}
        \item We can write $f(x) = x^{2g} + \sum_{i=1}^{g-1} a_i (x^{2g-i} + q^i x^{i})+a_g x^g + q^g$, where $(a_g, q) = 1$.
        \item There is an ordinary abelian variety $A$ over $\mathbb{F}_q$ such that $f(x)$ is the characteristic polynomial of $\Frob_q \in \End(A)$. 
    \end{enumerate}
\end{lemma}
\begin{proof}
    This follows from Honda--Tate theory -- see \cite{van25}. 
\end{proof}
    We call a polynomial satisfying the equivalent conditions of Lemma \ref{lem_honda_tate} an ordinary $q$-Weil polynomial of degree $2g$. Note that a monic polynomial of even degree is an ordinary $q$-Weil polynomial if and only if its irreducible factors are ordinary $q$-Weil polynomials. 
\begin{lemma}\label{lem_existence_of_Weil_polynomials}
    Let $q$ be a power of the prime $p$, and let $M \geq 1$ be an integer prime to $q$. Let 
    \[ a(x)  = x^g + a_1 x^{g-1} + \dots + a_g \in \mathbb{Z}_p[x] \]
    be a monic polynomial of degree $g$, with $a_g \in \mathbb{Z}_p^\times$, and let $b(x) \in (\mathbb{Z} / M \mathbb{Z})[X]$ be a monic polynomial of degree $2g$ of the form \[ b(x) = x^{2g} + \sum_{i=1}^{g-1} b_i (x^{2g-i} + q^i x^{i}) + b_g x^g + q^g. \] 
    Fix $n \geq 1$. Then for all sufficiently large $\ell \geq 1$ with $\ell \equiv 1 \text{ mod } \phi(M)$, we can find an ordinary $q^\ell$-Weil polynomial $f(x) \in \mathbb{Z}[x]$ of degree $2g$ satisfying the following conditions:
    \begin{enumerate}
        \item $f(x) \equiv x^g a(x) \text{ mod }p^{n}$.
        \item $f(x) \equiv b(x) \text{ mod }M$.
    \end{enumerate}
\end{lemma}
\begin{proof}
    We use the formalism of \cite{DiP98}. Let $V_g \subset \mathbb{R}^g$ denote the set of vectors $\underline{v} = (v_1, \dots, v_g)$ such that the polynomial 
    \[ g_{\underline{b}}(x) = x^{2g} + \sum_{i=1}^{g-1} v_i (x^{2g-i} +  x^{i}) + v_g x^g + 1 \in \mathbb{R}[x] \]
    has all of its roots in $\mathbb{C}$ on the unit circle. Then $V_g$ is compact, with Lipschitz-parameterizable boundary and non-empty interior (see \cite[Lemma 2.3.3]{DiP98}).

    For any $r > 0$, let $\phi_r : \mathbb{R}^g \to \mathbb{R}^g$ denote the map 
    \[ \phi_r : (v_1, \dots, v_g) \mapsto (r^{-1/2} v_1, \dots, r^{-g/2} v_g). \]
    Then $\phi_r^{-1}(V_g)$ is in bijective correspondence with the set of real polynomials with an even number of real roots, and all complex roots of absolute value $r$. 

    Let $\ell_0 \geq 1$ be such that $q^{\ell_0} \geq p^n$. Let $\underline{c} \in \mathbb{Z}^g$ be such that $\underline{c} \equiv (a_1, \dots, a_g) \text{ mod }p^n$ and $\underline{c} \equiv (b_1, \dots, b_g) \text{ mod }M$. If $\ell \geq \ell_0$, $\ell \equiv 1 \text{ mod }\phi(M)$, and $(f_1, \dots, f_g) \in (\underline{c} + p^n M \mathbb{Z}^g) \cap \phi_{q^\ell}^{-1}(V_g)$, then the polynomial 
    \[ f(x) = x^{2g} + \sum_{i=1}^{g-1} f_i (x^{2g-i} +  q^{\ell i} x^{i}) + f_g x^g + q^{\ell g} \]
    will satisfy the conditions of the lemma (because $q^\ell \equiv q \text{ mod }M$). 

    To complete the proof, we therefore need to show that when $\ell$ is sufficiently large, $\phi_{q^\ell}( \underline{c} + p^n M \mathbb{Z}^g) \cap V_g$ is non-empty. By \cite[Theorem 2.4]{Wid12}, there is a constant $c(V_g) > 0$ such that for any lattice coset $\underline{t} + \Lambda \subset \mathbb{R}^g$, we have
    \[ \left| | V_g \cap (\underline{t} + \Lambda) | - \frac{\operatorname{vol}(V_g)}{\operatorname{covol}(\Lambda)} \right| \leq c(V_g) \max_{0 \leq i < n} (\lambda_1 \lambda_2 \dots \lambda_i)^{-1}. \]
    (The statement there is given for lattices, i.e.\ subgroups, of $\mathbb{R}^g$, but the case of cosets follows by translation of the set $S$ of \emph{loc. cit.} -- note that this does not change the constants associated to the Lipschitz parameterization of $S$.) The successive minima of $\phi_{q^\ell}(p^n M \mathbb{Z}^g)$ are $(p^n M) \cdot (q^{-g\ell/2}, q^{-(g-1)\ell/2}, \dots, q^{- \ell/2})$, so in our case we find
    \begin{multline*}    \left| |V_g \cap \phi_{q^\ell}( \underline{c} + p^n M \mathbb{Z}^g)| - \operatorname{vol}(V_g) q^{\ell g(g+1)/4} (p^n M)^{-g} \right| \\ \leq c(V) (p^n M)^{1-g} q^{\ell (g(g+1)/4 - 1/2)}. 
    \end{multline*}
    The volume $\operatorname{vol}(V_g)$ is positive (and indeed is computed precisely in \cite{DiP98}). We see that as soon as $\ell$ is large enough, the desired intersection must indeed be non-empty. 
\end{proof}
Let $F$ be a totally real number field, and let $E$ be a coefficient field. Let $v$ be a $p$-adic place of $F$, and let $M$ be a discrete $\cO[G_{F_v}]$-module, torsion as $\cO$-module, on which the action of $G_{F_v}$ is unramified. In this case, we define a subgroup
\[ H^1_f(F_v, M(1)) \leq H^1(F_v, M(1)) \]
as the pre-image under the map
\[ H^1(F_v, M(1)) \to H^1(I_{F_v}, M(1)) \cong (F_v^{ur})^\times \otimes_{\mathbb{Z}} M \]
of the subgroup $(\cO_{F_v}^{ur})^\times \otimes_{\mathbb{Z}} M$. If instead $M$ is an $\cO[G_{F_v}]$-module, finite free as $\cO$-module, such that for every $n \geq 1$, $M / \varpi^n M$ is a discrete $\cO[G_{F_v}]$-module on which the action of $G_{F_v}$ is unramified, then we define
\[ H^1_f(F_v, M(1)) = \varprojlim_n H^1_f(F_v, M(1) / \varpi^n M(1)) \leq H^1(F_v, M(1)). \]
\begin{lemma}\label{lem_lifting_ordinary_p_divisible_groups}
    Let $\alpha : G_{F_v} \to \cO^\times$ be a continuous unramified character. Then the natural homomorphism
    \[ H^1_f(F_v, \cO(\alpha \epsilon)) \to H^1_f(F_v, \cO / (\varpi^n)(\alpha \epsilon)) \]
    is surjective.
\end{lemma}
\begin{proof}
    This follows immediately from \cite[Lemma 2.4.2]{Kis09}.
\end{proof}
\begin{lemma}\label{lem_existence_of_Weil_numbers}
    Let $v$ be a $p$-adic place of $F$, let $E$ be a coefficient field, and let $\alpha \in \cO^\times$. Fix $n \geq 1$. Then we can find the following data:
\begin{enumerate}
    \item A CM number field $N$ such that, writing $N^+$ for its maximal totally real subfield, the extension $N / N^+$ is ramified at some finite place, and each $p$-adic place of $N^+$ splits in $N$.
    \item A prime $\mathfrak{p} | p \cO_N$, and a continuous isomorphism $E \cong N_\mathfrak{p}$.
    \item An ordinary $q_v$-Weil number $\beta_v \in N$ such that the images of $\alpha, \beta_v$ in $\cO / (\varpi^n) \cong \cO_{N_\mathfrak{p}} / (\varpi^n)$ are the same. 
\end{enumerate}
\end{lemma}
\begin{proof}
    Let $f_\alpha(x) \in \mathbb{Z}_p[x]$ be the minimal polynomial of $\alpha$ over $\mathbb{Z}_p$, and let $g$ be its degree. Let $k = \ord_\varpi f'_\alpha(\alpha)$. We are free to increase $n$, so let's assume that $n > k$, and set $m = n + k$. Let $C$ denote the cardinality of $(\cO / \varpi^m \cO)^\times$. We can choose distinct primes $r_1, r_2 \neq p$ with the following properties:
    \begin{itemize}
        \item There exists an irreducible polynomial $R_1(x) \in \mathbb{F}_{r_1}[x]$, monic of degree $g$, such that $x^g R_1(x + q_v /  x) \in \mathbb{F}_{r_1}[x]$ is separable.
        \item There exist $s_2 \in \mathbb{Z}$ such that the $r_2$-adic valuation of the discriminant of $x^2 - s_2 x + q_v$ is $1$; and $R_2(x) \in \mathbb{Z} / r_2^2 \mathbb{Z}[x]$, monic of degree $g$, such that $R_2(x) \text{ mod }r_2$ is separable and $R_2(s_2) \equiv 0 \text{ mod }r_2^2$. 
    \end{itemize}
    Let $M = r_1 r_2^2$. By Lemma \ref{lem_existence_of_Weil_polynomials}, we can find an odd prime $\ell$ such that $\ell \equiv 1 \text{ mod }C\phi(M)$, $p^m | q_v^\ell$, and there exists an ordinary $q_v^{\ell}$-Weil polynomial $f(x) \in \mathbb{Z}[x]$ of degree $2g$ such that $f(x) \equiv x^g f_\alpha(x) \text{ mod }p^m$, $f(x) \equiv x^g R_1(x + q_v / x) \text{ mod }r_1$, and $f(x) \equiv x^g R_2(x + q_v / x) \text{ mod }r_2^2$. 

    The congruence modulo $r_1$ implies that $f(x)$ is irreducible; let $N_1 = \mathbb{Q}[x] / (f(x))$, and let $\beta_1 = x + (f(x))$. Then $N_1$ is a CM field. Writing $N_1^+ = \mathbb{Q}(\beta_1 + q_v^\ell / \beta_1)$ for its maximal totally real subfield, we see that $r_2$ is unramified in $N_1^+$, but ramified in $N_1$, and each $p$-adic place of $N_1^+$ splits in $N_1$ (because $\beta_1$ is ordinary). The congruence modulo $p^m$ implies that we have $f(\alpha) \in \varpi^m \cO$, so by Hensel's Lemma, there exists a unique $a \in \cO$ such that $f(a) = 0$, and $a \equiv \alpha \text{ mod }\varpi^{k+1}$; and in fact, we have $a \equiv \alpha \text{ mod }\varpi^{n}$ (since $n = m-k$). In other words, there is an embedding $N_1 \to E$ such that the image of $\beta_1$ in $\cO / \varpi^n \cO$ is $\alpha \text{ mod }\varpi^n$. 

    We now consider the polynomial $x^\ell - \beta_1 \in N_1[x]$. We split into cases according to whether or not it has a root in $N_1$. If it has a root, then there is an element $\beta \in N_1$ such that $\beta^\ell = \beta_1$. This $\beta$ is an ordinary $q_v$-Weil number. Our choice of $\ell$ implies that $\cO^\times$ is uniquely $\ell$-divisible, and that raising to the $\ell$th power acts as the identity on $(\cO / \varpi^n \cO)^\times$. Thus the image of $\beta$ in $\cO / \varpi^n \cO$ is again equal to $\alpha \text{ mod }\varpi^n$. To complete the proof, we take $N$ to be the compositum of $N_1$ with a totally real extension $N_2 / \mathbb{Q}$, linearly disjoint from $N_1$, which is split at $r_2$ and such that, for each $p$-adic place $v$ of $N_2$, there is an isomorphism $N_{2, v} \cong E$. Then $N / N^+$ is ramified at some $r_2$-adic place, each $p$-adic place of $N^+$ splits in $N$, $\beta_v = \beta \in \cO$ is an ordinary $q_v$-Weil number, and we can choose an extension of the embedding $N_1 \to E$ to an embedding $N \to E$ with dense image. The proof is complete on taking $\mathfrak{p}$ to be the prime of $\cO_N$ induced by this embedding. 

    If $x^\ell - \beta_1$ does not have a root in $N_1$, then it is irreducible (as $\ell$ is an odd prime). In this case, we take $N_3 / N_1$ to be the extension determined by this polynomial, and write $\beta \in N_3$ for a root, which is an ordinary $q_v$-Weil number. For the same reasons as in the previous paragraph, there is an embedding $N_3 \to E$ such that the image of $\beta$ in $\cO / (\varpi^n)$ is equal to $\alpha \text{ mod }\varpi^n$. We claim that there is an $r_2$-adic place of $N_3^+$ that is ramified in $N_3$. To see this, let us first consider the diagram of field extensions
    \[ \xymatrix{ N_3 = \mathbb{Q}(\beta) \ar@{-}[r]^-2 \ar@{-}[d]^\ell& N_3^+ = \mathbb{Q}(\beta + q_v / \beta) \ar@{-}[d]^\ell \\ N_1 = \mathbb{Q}(\beta_1) \ar@{-}[r]^-2 & N_1^+ = \mathbb{Q}(\beta_1 + q_v^\ell / \beta_1).} \]
    The labels indicate the degree of each extension. We next adjoin $\zeta_\ell$ to each field, to obtain a diagram
    \[ \xymatrix{ \mathbb{Q}(\beta, \zeta_\ell) \ar@{-}[r]^-2 \ar@{-}[d]^\ell& \mathbb{Q}(\beta + q_v / \beta, \zeta_\ell) \ar@{-}[d]^\ell \\ \mathbb{Q}(\beta_1, \zeta_\ell) \ar@{-}[r]^-2 &  \mathbb{Q}(\beta_1 + q_v^\ell / \beta_1, \zeta_\ell).} \]
    (The left extension has degree $\ell$ because $\ell$ is prime to $[ \mathbb{Q}(\zeta_\ell) : \mathbb{Q} ]$. The bottom extension has degree 2 because there is an $r_2$-adic prime ramified in $N_1 / N_1^+$, so $N_1$ cannot be contained in $N_1^+(\zeta_\ell)$.) The extension $\mathbb{Q}(\beta, \zeta_\ell) / \mathbb{Q}(\beta_1 + q_v^\ell / \beta_1, \zeta_\ell)$ is Galois, with dihedral Galois group generated by $\sigma(\beta) = q_v / \beta$ and $\rho(\beta) = \zeta_\ell \beta$. If we can show that the extension $\mathbb{Q}(\beta, \zeta_\ell) / \mathbb{Q}(\beta + q_v / \beta, \zeta_\ell)$ is ramified, it will follow that $N_3 / N_3^+$ is ramified, as required.

    However, by construction, there is an $r_2$-adic place $v_1$ of $N_1^+(\zeta_\ell)$ whose ramification index in $N_3(\zeta_\ell)$ is equal to 2. Therefore, any inertia group associated to $v_1$ is cyclic of order 2. All involutions in the dihedral group $D_{2\ell}$ are conjugate, so it follows that we can find a place $v_3$ of $N_3(\zeta_\ell)$ lying above $v_1$ such that $I_{v_3 / v_1} = \Gal( N_3(\zeta_\ell) / N_3^+(\zeta_\ell))$. In particular, the $r_2$-adic place of $N_3^+(\zeta_\ell)$, given by the restriction of $v_3$ to this field, ramifies in $N_3(\zeta_\ell)$.

    To conclude the proof in this case, choose a totally real extension $N_2 / \mathbb{Q}$, split at $r_2$ and linearly disjoint from $N_3 / \mathbb{Q}$, such that for each $p$-adic place $v$ of $N_2$, there is an isomorphism $N_{2, v} \cong E$. We take $N = N_3 \cdot N_2$, $\beta_v = \beta$, and choose $\mathfrak{p}$ by extending the given embedding $N_3 \to E$ to an embedding $N \to E$ with dense image, and taking $\mathfrak{p}$ to be the corresponding prime of $\cO_N$. 
\end{proof}
\begin{proposition}\label{prop_existence_of_AV}
    Let $v$ be a $p$-adic place of $F$, let $E$ be a coefficient field, and let $\rho : G_{F_v} \to \GL_2(\cO)$ be a continuous, ordinary representation with Hodge--Tate weights $\{0, 1 \}$, with $\det \rho = \epsilon^{-1}$, and which is moreover crystalline. Fix $n \geq 1$. Then we can find the following data:
    \begin{enumerate}
        \item A totally real number field $M$.
        \item A prime $\mathfrak{p} | p \cO_M$, and an isomorphism $M_{\mathfrak{p}} \cong E$. 
        \item A polarised $M$-HBAV $(X, m, j)$ over $F_v$, which has good ordinary reduction, together with an isomorphism $X[\mathfrak{p}^n]^\vee \cong \rho \otimes_\cO \cO / (\varpi^n)$ of $\cO[G_{F_v}]$-modules.
    \end{enumerate}
\end{proposition}
\begin{proof}
    There is an isomorphism
    \[ \rho \sim \left( \begin{array}{cc} \alpha & \ast \\ 0 & \epsilon^{-1} \alpha^{-1} \end{array} \right) \]
    for an unramified character $\alpha : G_{F_v} \to \cO^\times$, and the extension class is defined by an element of $H^1_f(F_v, \cO(\alpha^2 \epsilon))$. Let $\alpha_v = \alpha(\Frob_v)$. By Lemma \ref{lem_existence_of_Weil_numbers}, we can find a CM  number field and an ordinary $q_v$-Weil number $\beta_v \in N$ with the following properties:
    \begin{itemize}
        \item Let $M$ denote the maximal totally real subfield of $N$. Then each $p$-adic place of $M$ splits in $N$, and there exists a finite place of $M$ that is ramified in $N$.
        \item There is a prime $\mathfrak{p} | p \cO_N$ and an isomorphism $N_{\mathfrak{p}} \cong E$ (and hence $\cO_N / \mathfrak{p}^n \cong \cO / (\varpi^n)$) sending $\beta_v \text{ mod }\mathfrak{p}^n$ to $\alpha_v \text{ mod }\varpi^n$.
    \end{itemize}
    By Honda--Tate theory, we can find an ordinary abelian variety $X$ over $k(v)$ and an embedding $N \to \End(X) \otimes \mathbb{Q}$ sending $\beta_v$ to $\Frob_v$. After replacing $X$ by an isogenous abelian variety, we can assume that this embedding comes from an isomorphism $\cO_N \to \End(X)$ (cf. \cite[Theorem 3.13]{Wat69}). Then $X$ is an $M$-HBAV over $k(v)$.  After again passing to an isogenous abelian variety, we can assume that there exists an isomorphism $j : \mathfrak{d}^{-1}_{M / \mathbb{Q}} \to \mathcal{P}(X, m)$ of ordered sheaves of $\cO_M$-modules (the argument is given in \cite[Proof of Lemma 1.2]{Tay02}, and uses the existence of a ramified prime in the quadratic extension $N / M$). Thus $(X, m, j)$ is a polarised $M$-HBAV.

    By Serre--Tate theory, and \cite[Proposition 1.9]{Rap78}, lifting $X$ to a polarised $M$-HBAV over $\cO_{F_v}$ is equivalent to lifting its $p$-divisible group with $\cO_M$-structure $T_p X$. Decomposing $T_p X = \oplus_{\mathfrak{q} | p \cO_M} T_{\mathfrak{q}} X$, we can lift the components $T_{\mathfrak{q}} X$ for $\mathfrak{q} \neq \mathfrak{p}$ arbitrarily; we need to be more careful with $T_{\mathfrak{p}} X$. Let $\beta : G_{F_v} \to \cO^\times$ be the unramified character sending $\Frob_v$ to $\beta_v$, let $c_\rho \in H^1_f(F_v, \cO(\alpha^2 \epsilon))$ be the extension class specified by $\rho$, and let 
    \[ c_{\rho, n} \in H^1_f(F_v, \cO / \varpi^n \cO(\alpha^2 \epsilon)) = H^1_f(F_v, \cO / \varpi^n \cO(\beta^2 \epsilon)) \]
    be its image under reduction modulo $\varpi^n$. By Lemma \ref{lem_lifting_ordinary_p_divisible_groups}, we can choose a class $c' \in H^1(F_v, \cO (\beta^2 \epsilon))$ that also maps to $c_{\rho, n}$ under reduction modulo $\varpi^n$. We take  $T_{\mathfrak{p}} X$ to be the extension of $(E / \cO)(\beta)$ by $(E / \cO)(\beta)^D$ specified by $c'$. The corresponding lift of $(X, m, j)$ then fulfills the requirements of the proposition.
\end{proof}
We can now give the proof of Theorem \ref{thm_pot_mod}.
\begin{proof}[Proof of Theorem \ref{thm_pot_mod}]
    We are free to increase $n$. We can therefore assume that $n$ is chosen so that, if $E' / E$ denotes the compositum of all quadratic extensions in $\overline{\mathbb{Q}}_p$, then $\ker(\GL_2(\cO') \to \GL_2(\cO' / \varpi^n \cO'))$ has no non-trivial elements of finite order. 

    Let us write $V$ for the rank 2 $\cO[G_F]$-module which is the  $\cO$-linear dual of $\rho$, and fix an isomorphism $e_\rho : \wedge^2 V \cong \cO(1)$. For each place $v | p$,
    Proposition \ref{prop_existence_of_AV} gives us the existence of the following data:
    \begin{itemize}
        \item A totally real number field $M_v$, a prime $\mathfrak{p}_v | p \cO_{M_v}$, and an isomorphism $M_{v, \mathfrak{p}_v} \cong E$.
        \item A polarised $M_v$-HBAV $(X'_v, m'_v, j'_v)$ over $F_v$, which has good ordinary reduction, together with an isomorphism
        $X'_v[\mathfrak{p}_v^n] \cong V / (\varpi^n)$ of $\cO[G_{F_v}]$-modules.
    \end{itemize}
    Let $M'$ denote the  compositum of the fields $M_v$ ($v | p$) inside $E$; then $M'$ comes equipped with a prime $\mathfrak{p}' | p \cO_{M'}$ lying above each $\mathfrak{p}_v$, and an isomorphism $M'_{\mathfrak{p}'} \cong E$. We choose a finite totally real extension $M / M'$ and a prime $\mathfrak{p} \leq \cO_M$ lying above $\mathfrak{p}'$ such that $\cO_{M'_{\mathfrak{p'}}}^{\times} \subset (\cO_{M_\mathfrak{p}}^\times)^2$. Lemma \ref{lem_scalar_extension_and_polarisations} shows that, if $v |p$ and we write $(X_v, m_v, j_v)$ for the triple arising from $(X'_v, m'_v, j'_v)$ by scalar extension, then $(X_v, m_v, j_v)$ is a polarised $M$-HBAV over $F_v$, and there is an isomorphism $\eta_v : X_v[\mathfrak{p}^n] \cong V \otimes_\cO \cO_{M_\mathfrak{p}} / (\varpi^n)$ of $\cO_{M_\mathfrak{p}} / (\varpi^n)[G_{F_v}]$-modules. Multiplying $\eta_v$ by an element of $\cO_{M_\mathfrak{p}}^\times$, we can (and do) assume that $\eta_v$ identifies $e_\rho$ with the isomorphism $e_v : \wedge^2 X_v[\mathfrak{p}^n] \cong \cO_{M_\mathfrak{p}} / (\varpi^n)(1)$ determined by the $j_v$-Weil pairing. 

    We next claim that for each $v \in \Sigma$, we can find a polarised $M$-HBAV $(X_v, m_v, j_v)$ over $F_v$ with the following properties:
    \begin{itemize}
        \item $X_v$ is semistable, with purely toric reduction.
        \item There is an isomorphism $\eta_v : X_v[\mathfrak{p}^n] \cong V \otimes_\cO \cO_{M_\mathfrak{p}} / (\varpi^n)$ of $\cO_{M_\mathfrak{p}} / (\varpi^n)[G_{F_v}]$-modules that identifies $e_\rho$ with the isomorphism $e_v$ determined by the $j_v$-Weil pairing. 
    \end{itemize}
    By the same argument as in the good reduction case, it suffices to find a polarised $M'$-HBAV $(X'_v, m'_v, j'_v)$ over $F_v$ with toric semistable reduction, such that there is an isomorphism $X'_v[\mathfrak{p}'^n] \cong V / (\varpi^n)$ of $\cO[G_{F_v}]$-modules. We do this using the analytic theory of \cite[\S 2]{Rap78}. The character $\alpha$ satisfies $\alpha^2 = 1$; if it is non-trivial, then we can make an unramified quadratic twist to reduce to the case that it is trivial. If $q \in \cO_{M'} \otimes_\mathbb{Z} F_v^\times$, then $\ord_v(q) \in \cO_{M'}$. If $\ord_v(q)$ is totally positive, then the algebraization $X'$ of the rigid analytic variety $X'^{an} = \cO_{M'} \otimes \mathbb{G}_m / \cO_{M'} \cdot q$ is a polarised $M'$-HBAV over $F_v$ with toric semistable reduction, and there is an isomorphism
    \[ T_{\mathfrak{p}'} X' \sim \left( \begin{array}{cc} \epsilon & \ast \\ 0 & 1 \end{array} \right), \]
    where the extension class is the image of $q$ in 
    \[ H^1(F_v, \cO_{M'_{\mathfrak{p}'}}(1)) \cong \cO_{M'_{\mathfrak{p}'}} \otimes_{\widehat{\mathbb{Z}}} \widehat{F_v^\times}. \]
    Let $c_{\rho, v} \in H^1(F_v, \cO(1))$ denote the class determined by $\rho|_{G_{F_v}}$, and let $c_{\rho, v, n}$ denote its image in $H^1(F_v, (\cO / \varpi^n \cO)(1))$. To construct an $X'$ with the desired properties, it is therefore enough to show that we can lift $c_{\rho, v, n}$ to an element $q \in \cO_{M'} \otimes_\mathbb{Z} F_v^\times$ such that $\ord_v(q) \in \cO_{M'}$ is totally positive. This is true. 

    We now consider the moduli space $Y$ over $F$, classifying for an $F$-scheme $R$ the set of isomorphism classes of tuples $(X, m, j, \eta)$, where $(X, m, j)$ is a polarised $M$-HBAV over $R$, and $\eta : X[\mathfrak{p}^n] \to V_R \otimes_\cO \cO_{M_\mathfrak{p}} / (\varpi^n)$ is an isomorphism of \'etale sheaves of $\cO_{M_\mathfrak{p}} / (\varpi^n)$-modules that identifies $e_{\rho}$ with the isomorphism determined by the $j$-Weil pairing. Such tuples have no non-trivial automorphisms, by construction, so $Y$ is a smooth algebraic space over $F$, that is geometrically connected  (cf. \cite[Th\'eor\`eme 1.28]{Rap78}). In fact, taking $S = \Sigma \cup \{ v | p \}$, $Y$ can naturally be extended to a smooth, surjective morphism $Y \to \Spec \cO_{F, S}$ (by extending the moduli problem in the obvious way). Moreover, $Y$ is a quasi-projective scheme, and it admits an open embedding $Y \to \overline{Y}$, where $\overline{Y}$ is a flat, projective scheme over $\cO_{F, S}$, with the property that the complement $\overline{Y} - Y$ is finite over $\cO_{F, S}$; in particular, the complement $\overline{Y}_F - Y_F$ has codimension $\dim Y_F = [M : \mathbb{Q}] \geq 2$ in $\overline{Y}_F$. (For the quasi-projectivity of $Y$, see e.g.\ the discussion in \cite[\S 2.2.1]{Dia23}; for the structure of the minimal compactification, see \cite[\S 6.4]{Dia23}.)

    By Theorem \ref{thm_moret-bailly}, we can find a finite totally real extension $F' / F$ and a point $P \in Y(\cO_{F', S})$ such that for each $v \in S$, $v$ splits in $F'$ and the image of $P$ in $Y(F'_w)$ (for any place $w | v$ of $F'$) approximates the tuple $(X_v, m_v, j_v, \eta_v)$ to any desired degree of precision. In particular, we can suppose that $P$ corresponds to an $M$-HBAV $(X, m, j)$ over $F'$ with the following properties:
    \begin{itemize}
        \item $X$ has good reduction outside $\Sigma$.
        \item For each place $w | p$ of $F'$, $X_{F'_w}$ has good ordinary reduction.
        \item For each place $w$ of $F'$ lying above $\Sigma$, $X_{F'_w}$ has toric semistable reduction.
        \item There is an isomorphism $X[\mathfrak{p}^n] \cong V \otimes_\cO \cO_{M_\mathfrak{p}} / (\varpi^n)$ of $\cO_{M_\mathfrak{p}} / (\varpi^n)[G_{F'}]$-modules.
    \end{itemize}
    By a standard argument, we can moreover require that $F' / F$ is linearly disjoint from any given finite extension of $F$ (see e.g.\ \cite[Proposition 3.1.1]{Bar14}). Let $L_\infty / F$ be the extension cut out by $(\rho \times \epsilon)$, and let $L / F$ be the compositum of all subextensions $L_\infty / K / F$ such that $K / F$ is a finite Galois extension with $\Gal(K / F)$ simple. Then $L / F$ is finite (being the compositum of a finite number of finite extensions), and if $F ' /F$ is linearly disjoint from $L / F$, then $(\rho \times \epsilon)(G_{F'}) = (\rho \times \epsilon)(G_F)$. This completes the proof. 
\end{proof}

\section{Modularity}
Let $F$ be a totally real number field, let $p$ be a prime, let $E_0 / \mathbb{Q}_p$ be a coefficient field with ring of integers $\cO_0$, and let $\rho : G_F \to \GL_2(\cO_0)$ be a continuous representation satisfying the following conditions:
\begin{itemize}
    \item $\det \rho = \epsilon^{-1}$, and $\rho$ is unramified at all but finitely many places.
    \item For each place $v | p$ of $F$, $\rho|_{G_{F_v}}$ is crystalline and ordinary, with Hodge--Tate weights $\{ 0, 1 \}$ with respect to any embedding $F_v \hookrightarrow \overline{\mathbb{Q}}_p$. 
    \item For every finite place $v$ of $F$, $\rho|_{G_{F_v}}$ is semistable.
    \item The projective image of $\rho|_{G_{F(\zeta_{p^\infty})}} \otimes_{\cO_0} E_0$ in $\PGL_2$ is Zariski dense. 
    \item $E_0$ is large enough, in the sense that every $p$-adic embedding of $F$ takes values in $E_0$.
\end{itemize}
In this section, we will state and prove a modularity lifting theorem for $\rho$ -- or more precisely, a theorem asserting the modularity of $\rho$, given the existence of an auxiliary modular Galois representation approximating $\rho$ to a high enough degree of ($p$-adic) precision. See Theorem \ref{thm_modularity_by_close_approximation} below. Before giving the statement, we need to introduce some notation.

Let $E_1 / E_0$ denote the compositum of all quadratic extensions of $E_0$ in $\overline{\mathbb{Q}}_p$, and let $E / E_1$ denote the compositum of all quadratic extensions of $E_1$. Let $\cO_1 \leq E_1$ and $\cO \leq E$ denote the respective rings of integers. 

Let $\Sigma$ denote the set of places such that $\mathrm{WD}(\rho|_{G_{F_v}})$ is ramified, and let $S_p$ denote the set of $p$-adic places of $F$. For any set $S$ of finite places of $F$ containing $\Sigma \cup S_p$, we introduce the ring $R_{\overline{D}, S} \in \CNL_\cO$, the universal pseudodeformation ring of the group determinant $\overline{D} = \det(X-\overline{\rho})$, classifying pseudodeformations of determinant $\epsilon^{-1}$.

We write $W_\cO = \operatorname{ad}^0 \rho \otimes_{\cO_0} \cO$ for the adjoint representation on traceless matrices, $W_m = W_\cO / (\varpi^m)$, and $W_{E / \cO} = (W_\cO \otimes_\cO E) / W_\cO$. We write $W_\cO^\ast = \Hom_\cO(W_\cO, \cO(1))$ for the Tate dual, and define $W_m^\ast$, $W_{E / \cO}^\ast$ similarly. 

Let $\overline{\rho} : G_F \to \GL_2(k)$ denote the reduction of $\rho$ modulo $\varpi$. We fix local lifting rings for each $v \in \Sigma \cup S_p$ as follows:
\begin{itemize}
    \item If $v \in S_p$, then $R_v$ classifies lifts of $\overline{\rho}|_{G_{F_v}}$ of determinant $\epsilon^{-1}$ which are crystalline with Hodge--Tate weights in $[0, 1]$. 
    \item If $v \in \Sigma$, then $R_v$ classifies all lifts of determinant $\epsilon^{-1}$. 
\end{itemize}
We write $R_v^\square$ for the unrestricted (determinant $\epsilon^{-1}$) lifting ring in each case, so $R_v$ is a quotient of $R_v^\square$. If $v \in \Sigma \cup S_p$, then $\rho|_{G_{F_v}}$ determines a map $R_v \to \cO$ of kernel $P_v$; we write $P_v^\square$ for its pre-image in $R_v^\square$. We define constants associated to these objects as follows:
\begin{itemize}
    \item $C_0$ is the exponent of $H^0(F, W_{E / \cO}) \oplus H^0(F, W_{E / \cO}^\ast)$.
    \item $C_1 \geq 0$ is the minimal integer such that $\operatorname{ad}^0 \rho(\cO[G_{F(\zeta_{p^\infty})}])$ contains $\varpi^{C_1} \End_\cO(W_\cO)$.
    \item $C_2 \geq 0$ is the exponent of $H^1(L_\infty / F, W_{E / \cO}^\ast)$, where $L_\infty / F$ is the extension cut out by $\rho \times \epsilon$. (This group is finite because $H^1(L_\infty / F, W_E^\ast) = 0$, e.g. by the argument of \cite[Lemma 4.28]{All23}.)
    \item $C_3 \geq 0$ is the integer $C(\rho(G_{F, S}))$ associated to the subgroup $\rho(G_{F, S}) \leq \GL_2(\cO)$ by Proposition \ref{prop_deforming_a_pseudocharacter}. 
    \item $C_4 \geq 0$ is the smallest integer such that there exists $\sigma \in G_{F(\zeta_{p^\infty})}$ such that the eigenvalues $\alpha, \beta \in \cO$ of $\rho(\sigma)$ satisfy $\ord_\varpi(\alpha - \beta)^2 = C_4$.
    \item $C_5 \geq 0$ is the smallest integer such that there exists $\sigma \in G_F$ such that $\epsilon(\sigma) \equiv 1 \text{ mod }\varpi^{C_5+1}$, $\epsilon(\sigma) \not\equiv 1 \text{ mod }\varpi^{C_5+2}$, and the eigenvalues $\alpha, \beta \in \cO$ of $\rho(\sigma)$ satisfy $\ord_\varpi(\alpha-\beta)^2 \leq C_5$. (To see that an integer with this property exists, choose $\sigma_0 \in G_{F(\zeta_{p^\infty})}$ such that $\rho(\sigma_0)$ is regular semi-simple, with characteristic polynomial discriminant of valuation $C_6 \geq 0$. Since $\epsilon$ is non-trivial on open subgroups of $G_F$, there exists $\sigma_1$ such that $(\rho \times \epsilon)(\sigma_0)  \equiv (\rho \times \epsilon)(\sigma_1) \text{ mod }\varpi^{C_6+1}$ and $\epsilon(\sigma_1) \neq 1$.)
    \item If $v \in \Sigma \cup S_p$, then $C_v$ is the exponent of the torsion subgroup of 
    \[ H^0(F_v, W_{E / \cO})^\vee \oplus H^0(F_v, W_{E / \cO}^\ast)^\vee  \oplus P_v / P_v^2 \oplus \ker( P_v^\square / (P_v^\square)^2 \to P_v / P_v^2). \] 
\end{itemize}
We define $C(\rho) = \max(C_0, C_1, C_2, C_3, C_4, 3C_5 + 1, \{ C_v \}_{v \in \Sigma \cup S_p}, \ord_\varpi 2)$. 

We can now state the main theorem of this section: 
\begin{theorem}\label{thm_modularity_by_close_approximation}
    Suppose we can find a continuous Galois representation $\rho' : G_F \to \GL_2(\cO_1)$ satisfying the following conditions:
    \begin{enumerate}
        \item $\rho'$ is modular: there exists an isomorphism $\iota : \overline{\mathbb{Q}}_p \to \mathbb{C}$ and a cuspidal, automorphic, ordinary regular algebraic automorphic representation $\pi$ of $\GL_2(\mathbb{A}_F)$ of weight 0 and trivial central character such that $\rho' \otimes_{\cO_1} \overline{\mathbb{Q}}_p \cong r_{\pi, \iota}$.
        \item There is an isomorphism $\rho \text{ mod }\varpi^{80 C(\rho)+1} \cong \rho' \text{ mod }\varpi^{80C(\rho)+1}$.
        \item For each finite place $v$ of $F$, $\rho'|_{G_{F_v}}$ is semistable; moreover, $\mathrm{WD}(\rho'|_{G_{F_v}})$ is ramified if and only if $\mathrm{WD}(\rho|_{G_{F_v}})$ is ramified (equivalently, if and only if $v \in \Sigma$).
        \item For each place $v | p$ of $F$, the unit eigenvalues of $\Frob_v$ on $\mathrm{WD}(\rho|_{G_{F_v}})$ and on $\mathrm{WD}(\rho'|_{G_{F_v}})$ have the same image in $k$. 
        \item $|\Sigma|$ and $[F : \mathbb{Q}]$ are even. 
    \end{enumerate}
    Then $\rho$ is also modular. 
\end{theorem}
The point here is that the constant $C(\rho)$ depends only on the group $(\rho \times \epsilon)(G_F)$ and the local representations $\rho|_{G_{F_v}}$ at ramified places, in a sense made precise by the following lemma, whose proof is immediate:
\begin{lemma}
    Let $F' / F$ be a finite totally real extension such that $\Sigma \cup S_p$ is split in $F'$ and $(\rho \times \epsilon)(G_F) = (\rho \times \epsilon)(G_{F'})$. Then $C(\rho) = C(\rho|_{G_{F'}})$.
\end{lemma}
This makes Theorem \ref{thm_modularity_by_close_approximation} suitable for combination with the results of \S \ref{sec_application_of_moret-bailly}. 

We now begin the proof of Theorem \ref{thm_modularity_by_close_approximation}. Using the Chebotarev density theorem, we can choose a place $u \not\in \Sigma \cup S_p$ of $F$ such that $q_u \equiv 1 \text{ mod }\varpi^{C_5+1}$, $q_u \not\equiv 1 \text{ mod }\varpi^{C_5+2}$, and the eigenvalues $\alpha_u, \beta_u \in \cO$ of $\rho(\Frob_u)$ satisfy $\ord_\varpi(\alpha_u - \beta_u)^2 \leq C_5$. We record some properties of this place.
\begin{lemma}
    We have $H^0(F_u, W_E^\ast) = 0$. Moreover, the torsion subgroups of $H^0(F_u, W_{E / \cO})^\vee$ and $H^0(F_u, W_{E / \cO}^\ast)^\vee$ are each annihilated by $\varpi^{C(\rho)}$. 
\end{lemma}
\begin{proof}
    The Jordan--H\"older factors of $W_E^\ast$ are $\mathrm{ur}_{ q_u^{-1} \alpha_u / \beta_u}$, $\mathrm{ur}_{q_u^{-1}}$, and $\mathrm{ur}_{\beta_u / \alpha_u q_u^{-1}}$. These have trivial invariants, since $\ord_\varpi (q_u - 1) = C_5+1$ and $\ord_\varpi (\alpha_u / \beta_u q_u^{\pm 1} - 1) \leq C_5$, showing that $H^0(F_u, W_E^\ast) = 0$, and moreover that $H^0(F_u, W_{E / \cO}^\ast)$ is $\varpi^{3 C_5+1}$-torsion. Similarly, we see that the torsion subgroup of $H^0(F_u, W_{E / \cO})^\vee$ is annihilated by $\varpi^{2C_5}$. 
\end{proof}
Now fix $S = \Sigma \cup S_p \cup \{ u \}$. Let $R_u^\square = R_u$ be the lifting ring classifying all lifts of $\overline{\rho}|_{G_{F_u}}$ of determinant $\epsilon^{-1}$. We define $R_\emptyset$ to be the quotient of $R_{\overline{D}, S}$ classifying pseudodeformations which are crystalline with Hodge--Tate weights in $[0, 1]$ at places $v \in S_p$ (this condition being defined for pseudodeformations as in \cite{Wak19}).

We next define sets of Taylor--Wiles places, and associated Selmer groups. Let $Q$ be a finite set of finite places of $F$, disjoint from $S$. A tuple $(Q, ( \alpha_v, \beta_v )_{v \in Q})$ is called a Taylor--Wiles datum of level $N \geq 1$ if it satisfies the following conditions:
\begin{itemize}
    \item For each $v \in Q$, $q_v \equiv 1 \text{ mod }p^N$.
    \item For each $v \in Q$, $\alpha_v, \beta_v \in \cO$ are the distinct eigenvalues of $\rho(\Frob_v)$, and $\ord_\varpi(\alpha_v-\beta_v)^2 = C_4$.
\end{itemize}
If $Q$ is part of a Taylor--Wiles datum, we call it a Taylor--Wiles set, and  define $R_Q$ to be the quotient of $R_{\overline{D}, S \cup Q}$ by the same set of relations as those defining $R_\emptyset$.
We write $\mathfrak{p}_Q$ for the kernel of the map $R_Q \to \cO$ classifying the pseudodeformation $\det(X - \rho)$.

For any $m \geq 1$ and $v \in S$, we define $\mathcal{L}_{v, m} \leq H^1(F_v, W_m)$ to be the subspace of cohomology classes $[\phi]$ such that the lifting $\rho_\phi = \rho(1 + \epsilon \phi) : G_{F_v} \to \GL_2(\cO \oplus \epsilon \cO / (\varpi^m))$ is of type $R_v$. Writing $\mathcal{L}_{v, m}^1$ for its pre-image in $Z^1(F_v, W_m)$, we have a canonical identification
\[ \mathcal{L}^1_{v, m} = \Hom_\cO( P_v / P_v^2, \cO / (\varpi^m)), \]
given essentially by $\phi \mapsto \rho_\phi$. We write $\mathcal{L}_{v, m}^\perp \leq H^1(F_v, W_m^\ast)$ for the orthogonal complement of $\mathcal{L}_{v, m}$ with respect to the Tate local duality pairing. In the presence of a Taylor--Wiles datum, we will define $\mathcal{L}_{v, m} = H^1(F_v, W_m)$ for $v \in Q$. We can thus define, for any $Q$, the Selmer groups
\[ H^1_{\mathcal{L}_{S \cup Q}}(W_m) = \ker( H^1(F_{S \cup Q} / F, W_m) \to \oplus_{v \in S \cup Q} H^1(F_v, W_m) / \mathcal{L}_{v, m}) \]
and
\begin{multline*} H^1_{\mathcal{L}_{S \cup Q}^\perp}(W_m^\ast) = \ker(H^1(F_{S \cup Q} / F, W_m^\ast) \\ \to \oplus_{v \in S \cup Q} H^1(F_v, W_m^\ast) / \mathcal{L}_{v, m}^\perp \bigoplus \oplus_{v \in S_\infty} H^1(F_v, W_m^\ast)). 
\end{multline*}
\begin{lemma}\label{lem_comparison_of_tangent_spaces}
    For any Taylor--Wiles set $Q$ of level $N \geq 1$, there is a morphism of $\cO$-modules $H^1_{\mathcal{L}_{S \cup Q}}(W_N) \to \Hom_\cO(\mathfrak{p}_Q / \mathfrak{p}_Q^2, \cO / \varpi^N)$, with kernel and cokernel annihilated by $\varpi^{4 C(\rho)}$.
\end{lemma}
\begin{proof}
    This is essentially the same as \cite[Proposition 2.15]{New23}, with the minor variation that we impose a crystalline condition at places $v \in S_p$. 
\end{proof}
\begin{lemma}\label{lem_dimensions_of_spaces_of_local_conditions}
    \begin{enumerate}
        \item If $v \in \Sigma \cup \{ u \}$, then $\dim_E (P_v / P_v^2) \otimes_\cO E = 3$.
        \item If $v \in S_p$, then $\dim_E (P_v / P_v^2) \otimes_\cO E = 3 + [F_v : \mathbb{Q}_p]$.
    \end{enumerate}
\end{lemma}
\begin{proof}
    The first part follows from \cite[Proposition 1.2.2]{All16}. For the second part, we argue as in the final paragraph of the proof of \cite[Lemma 2.19]{New23} (using results of Kisin and Liu) to reduce to \cite[Theorem 1.2.7]{All16}. 
\end{proof}
\begin{lemma}\label{lem_TW_local_triviality}
    Let $C \geq C(\rho)$, let $N \geq 2C$, let $Q$ be a Taylor--Wiles set of level $N$, and let $v \in Q$. Then the maps $H^1(F_v, W_C^\ast) \to H^1(F_v, W_{2C}^\ast)$ and $H^1(F_v, W_C^\ast) \to H^1(F_v, W_N^\ast)$ have the same kernel.
\end{lemma}
\begin{proof}
    Let $A = H^0(W_N^\ast)$. Writing down the usual exact sequences, we see that we need to show that the map $A[\varpi^C] / \varpi^C A[\varpi^{2C}] \to A[\varpi^{N-C}] / \varpi^C A[\varpi^N]$ is an isomorphism. It is easy to see that this is true provided there is an isomorphism $A \cong \cO / (\varpi^N) \oplus A'$, where $\varpi^C A' = 0$. Since $Q$ has level $N$, there is a short exact sequence of $G_{F_v}$-modules
    \[ 0 \to \cO / (\varpi^N) \to W_N^\ast \to B \to 0, \]
    where $B$ itself sits in a short exact sequence
    \[ 0 \to \cO(\mathrm{ur}_{\alpha_v / \beta_v}) / (\varpi^N) \to B \to \cO(\mathrm{ur}_{\beta_v / \alpha_v}) / (\varpi^N) \to 0. \]
    Since $C \geq 2 \operatorname{ord}_{\varpi}(\alpha_v / \beta_v - 1)$, it follows that $\varpi^C H^0(F_v, B) = 0$; since $\cO / (\varpi^N)$ is injective as a module over itself, it follows that $A$ has the required form. This completes the proof. 
\end{proof}
\begin{lemma}\label{lem_H_1_is_2C_free}
    Suppose that $C \geq C(\rho)$, that $N > 2 C$, and that $Q$ is a Taylor--Wiles set of level $N+C$ such that $\varpi^{C} \Sha^1_{S \cup Q}(W_N^\ast) = 0$. Then:
    \begin{enumerate}
        \item For each $m$ such that $2C \leq m \leq N$, we have
        \[ \ell(\Sha^1_{S \cup Q}(W_m^\ast)) = \ell(\Sha^1_{S \cup Q}(W_N^\ast)). \]
        \item  $H^1(F_{S \cup Q} / F, W_N)$ is a $3C$-free $\cO / (\varpi^N)$-module of rank $2 [F : \mathbb{Q}] + |Q| + \sum_{v \in S} \dim_E H^0(F_v, W_E^\ast)$.
    \end{enumerate}    
\end{lemma}
Notation here is as follows: $\ell$ denotes length as $\cO$-module, and $h^i_v(-) = \ell(H^i(F_v, -))$. We define
\[ \Sha^1_{S \cup Q}(W_N^\ast) = \ker( H^1(F_{S \cup Q}/ F, W_N^\ast) \to \oplus_{v \in S \cup Q \cup S_\infty} H^1(F_v, W_N^\ast) ). \]
\begin{proof}
    For (1), we observe that for any $m \geq 1$, there is a short exact sequence
    \[ 0 \to H(m) \to \Sha^1_{S \cup Q}(W_m^\ast) \to \Sha^1_{S \cup Q}(W_{E / \cO}^\ast)[\varpi^m], \]
    where we define
    \[ H(m) = \ker(H^0(F, W_{E / \cO}^\ast) / (\varpi^m) \to \oplus_{v \in S \cup Q} H^0(F_v, W_{E / \cO}^\ast) / (\varpi^m)). \]
    When $m \geq C$, the groups $H^0(F, W_{E / \cO}^\ast)$ and $H^0(F_u, W_{E / \cO}^\ast)$ are annihilated by $\varpi^m$, so in fact $H(m) = 0$. (This is the reason for introducing $u$.) To prove (1), it is therefore enough to show that the image of $\Sha^1_{S \cup Q}(W_m^\ast)$ in $H^1(F, W_{E / \cO}^\ast)$ is independent of $m$ for $2C \leq m \leq N$. Let $x_N \in \Sha^1_{S \cup Q}(W_N^\ast)$, and let $x_\infty$ denote its image in $\Sha^1_{S \cup Q}(W^\ast_{E / \cO})$. Then $\varpi^C x_\infty = 0$, by hypothesis, so we can find $y_C \in H^1(F_S / F, W_C^\ast)$ such that the image $y_\infty$ of $y_C$ in $H^1(F, W^\ast_{E / \cO})$ is equal to $x_\infty$. In particular, $x_N - y_N$ is contained in the group 
    \[ H^0(F, W_{E / \cO}^\ast) / (\varpi^N) \leq H^1(F_S / F, W_N^\ast). \]
    This group is annihilated by $\varpi^C$, so $x_{N+C} = y_{N+C}$. We claim that $y_{2C}$ is locally trivial, hence $y_{2C} \in \Sha^1_{S \cup Q}(W_{2C}^\ast)$. If $v \in Q$, then this is true by Lemma \ref{lem_TW_local_triviality}.  If $v \in S \cup S_\infty$, then this is true because $y_{C, v} \in H^1(F_v, W_C^\ast)$ lies in the group
    \[ \ker(H^1(F_v, W_C^\ast) \to H^1(F_v, W_{N+C}^\ast)), \]
    which may be identified with $H^0(F_v, W_N^\ast) / \varpi^C H^0(F_v, W_{N+C}^\ast)$. The definition of $C$ implies that we have
    \[ H^0(F_v, W_N^\ast) / \varpi^C H^0(F_v, W_{N+C}^\ast) = H^0(F_v, W_C^\ast) / \varpi^C H^0(F_v, W_{2C}^\ast),  \]
    and so $y_{2C, v} = 0$. We have shown that $\Sha^1_{S \cup Q}(W_{2C}^\ast)$ and $\Sha^1_{S \cup Q}(W_N^\ast)$ have the same image in $H^1(F, W_{E / \cO}^\ast)$, which proves (1).

    For (2), we use the following formula, which is a reformulation (using Poitou--Tate duality) of the global Euler characteristic formula (see \cite[p. 68]{Mil06}):
    \begin{multline*}  h^1(F_{S \cup Q / F}, W_m) = \ell(\Sha^1_{S \cup Q}(W_m^\ast)) + h^0(F, W_m) - h^0(F, W_m^\ast) \\ + \sum_{v \in S \cup Q} h^0_v(W_m^\ast)  + 3 m [F : \mathbb{Q}] + \sum_{v | \infty} \left( h^0_{v, T}(W_m^\ast) - h^0_v(W_m) \right). 
    \end{multline*}
    (If $v | \infty$ and $c_v \in G_F$ is a complex conjugation at $v$, then we write $h^0_{v, T}(W_m^\ast)$ for the length of the Tate cohomology group $h^0_{v, T}(W_m^\ast) = (W_m^\ast)^{c_v} / (1+c_v) W_m^\ast$. A simple computation shows that $h^0_{v, T}(W_m^\ast)$ is constant for $m \geq \ord_\varpi 2$, so in particular for $m \geq C$.) 
    It follows that for $C \leq m \leq N$, we have
     \begin{multline*} \ell(H^1(F_{S \cup Q / F}, W_{E / \cO})[\varpi^m]) = \ell(\Sha^1_{S \cup Q}(W_m^\ast)) - h^0(F, W_m^\ast) \\ + \sum_{v \in S \cup Q} h^0_v(W_m^\ast) + 3 m [F : \mathbb{Q}] + \sum_{v | \infty} \left( h^0_{v, T}(W_m^\ast) - h^0_v(W_m) \right). 
    \end{multline*}
    It follows from (1) and the definition of $C$ that when $2C \leq m \leq N-1$, we have
    \begin{multline*} \ell(H^1(F_{S \cup Q} / F, W_{E / \cO})[\varpi^{m+1}] - \ell(H^1(F_{S \cup Q} / F, W_{E / \cO})[\varpi^m]) \\ = \sum_{v \in S} \dim_E H^0(F_v, W_E^\ast) + |Q| + 2 [F : \mathbb{Q}]. \end{multline*}
    We conclude that there is an isomorphism
    \[ H^1(F_{S \cup Q / F}, W_{E / \cO})[\varpi^{N}] \cong (\cO / (\varpi^N))^{\sum_{v \in S} \dim_E H^0(F_v, W_E^\ast) + |Q| + 2 [F : \mathbb{Q}]} \oplus T(N), \]
    where $\varpi^{2C} T(N) = 0$. Using the short exact sequence
    \[ 0 \to H^0(F, W_{E / \cO}) / (\varpi^N) \to H^1(F_{S \cup Q} / F, W_N) \to H^1(F_{S \cup Q}/F, W_{E / \cO})[\varpi^N] \to 0, \]
    we conclude that $H^1(F_{S \cup Q}/F, W_N)$ is $3C$-free of rank $\sum_{v \in S} \dim_E H^0(F_v, W_E^\ast) + |Q| + 2 [F : \mathbb{Q}]$, as required. 
\end{proof}
\begin{lemma}\label{lem_H_1_v_is_2C_free}
    Suppose that $N > 4 C(\rho)$. Then $\oplus_{v \in S} H^1_v(W_N) / \mathcal{L}_{v, N}$ is a $2C(\rho)$-free $\cO / (\varpi^N)$-module of rank $2 [ F : \mathbb{Q}] + \sum_{v \in S} \dim_E H^0(F_v, W_E^\ast)$.
\end{lemma}
\begin{proof}
   It suffices to show that for each $v \in S_p$, $H^1_v(W_N) / \mathcal{L}_{v, N}$ is $2C(\rho)$-free of rank $2 [ F_v : \mathbb{Q}_p] + \dim_E H^0(F_v, W_E^\ast)$. (The other places of $S$ contribute 0 to both sides of the proposed equality -- note that at a semistable ramified place $v$, $\mathrm{WD}(\rho|_{G_{F_v}})$ is necessarily generic.) Let $K_v = \ker(P_v^\square / (P_v^\square)^2 \to P_v / P_v^2)$. Then $\Hom_\cO(K_v, \cO / (\varpi^N))$ is a $C(\rho)$-free $\cO / (\varpi^N)$-module of rank 
   \[ \dim H^0(F_v, W_E^\ast) + 2 [F_v : \mathbb{Q}_p]. \]
   There's a short exact sequence
   \[ 0 \to H^1_v(W_N) / \mathcal{L}_{v, N} \to \Hom_\cO(K_v, \cO / (\varpi^N)) \to (P_v / (P_v)^2)[\varpi^N]^\vee. \]
   The final term here is $\varpi^{C(\rho)}$-torsion, so it follows that $H^1_v(W_N) / \mathcal{L}_{v, N}$ is $2C(\rho)$-free of the expected rank. 
\end{proof}
\begin{proposition}\label{prop_application_of_Poitou--Tate}
    Suppose that $C \geq C(\rho)$, $N > 24 C$, and that $Q$ is a Taylor--Wiles set of level $N + C$ such that $\varpi^{C} H^1_{\mathcal{L}^\perp_{S \cup Q}}(W_N^\ast) = 0$. Then $H^1_{\mathcal{L}_{S \cup Q}}(W_N)$ is a $12 C$-free $\cO / (\varpi^N)$-module of rank $|Q|$.
\end{proposition}
\begin{proof}
    We use the Poitou--Tate exact sequence 
    \[ 0 \to H^1_{\mathcal{L}_{S \cup Q}}(W_N) \to H^1(F_{S \cup Q} / F, W_N) \to \oplus_{v \in S} H^1_v(W_N) / \mathcal{L}_{v, N} \to H^1_{\mathcal{L}^\perp_{S \cup Q}}(W_N^\ast)^\vee. \]
    We have $\Sha^1_{S \cup Q}(W_N^\ast) \leq H^1_{\mathcal{L}_{S \cup Q}^\perp}(W_N^\ast)$, so in particular $\varpi^C \Sha^1_{S \cup Q}(W_N^\ast) = 0$. We can thus apply Lemma \ref{lem_H_1_is_2C_free} to conclude that the second term in this sequence is $3C$-free. Lemma \ref{lem_H_1_v_is_2C_free} shows that the third term is $2C$-free. We can make the sequence exact on the right by replacing the last term by a subgroup, annihilated by $\varpi^C$ by hypothesis, and therefore also $2C$-free (of rank $0$). We can therefore apply Proposition \ref{prop_euler_characteristic} to conclude that $H^1_{\mathcal{L}_{S \cup Q}}(W_N)$ is $12C$-free of rank $|Q|$, as required. 
\end{proof}
If $N \geq 1$, we now write $L_N / F$ for the extension of $F(\zeta_{p^\infty})$ cut out by $\rho \text{ mod }\varpi^N$. Then $L_\infty = \cup_{N \geq 1} L_N$.
\begin{lemma}\label{lem_annihilation_of_auxiliary_groups}
    Let $N \geq 1$.
    \begin{enumerate}
        \item We have $\varpi^{C(\rho)} H^0(F, W_N^\ast) = 0$.
        \item For any $\phi \in \Hom_{G_F}(G_{L_N}, W_N^\ast)$ and $0 \leq M \leq N-1$ such that $\varpi^M \phi \neq 0$, we have $\varpi^{N+C(\rho)-M-1} W_N^\ast \leq \cO \cdot \phi(G_{L_N})$.
        \item We have $\varpi^{2C(\rho)} H^1(L_N / F, W_N^\ast) = 0$.
    \end{enumerate}
\end{lemma}
\begin{proof}
    We write $C = C(\rho)$.
    
    (1) We have $H^0(F, W_N^\ast) \leq H^0(F, W_{E / \cO}^\ast)$, and the latter group is annihilated by $\varpi^C$ by definition of $C$.

    (2) Again by definition of $C$, we have $\operatorname{ad}^0 \rho(\cO[G_{F(\zeta_{p^\infty})}]) \supset \varpi^C \End_\cO(W_\cO)$. It follows that $\cO \cdot \phi(G_{L_N})$ contains $\varpi^{N+C-M-1} W_N^\ast$. 

    (3) It suffices to show that $\varpi^{2C} H^1(L_\infty / F, W_N^\ast) = 0$. This is true since, again by definition of $C$, we have 
    \[ \varpi^C H^0(F, W_{E / \cO}^\ast)  = \varpi^C H^1(L_\infty / F, W_{E / \cO}^\ast) = 0. \]
\end{proof}
\begin{lemma}\label{lem_killing_dual_Selmer}
    We can find an integer $q \geq 0$ with the following property: for any $N \geq 1$, and any $C \geq C(\rho)$, there exists a Taylor--Wiles set $Q$ of level $N+C$ with $|Q| = q$ such that $\varpi^{3C(\rho)} H^1_{Q\text{-split}}(F_S / F, W_N^\ast) = 0$.
\end{lemma}
Here we take 
\[ H^1_{Q\text{-split}}(F_S / F, W_N^\ast) = \ker( H^1(F_S / F, W_N^\ast) \to \oplus_{v \in Q} H^1(F_v, W_N^\ast)). \]
\begin{proof}
    Let $q = \dim_k H^0(F, W^\ast_{E / \cO})[\varpi] + \dim_k H^1(F_S / F, W^\ast_{E / \cO})[\varpi]$. Then, for any $N \geq 1$, $H^1(F_S / F, W_N^\ast)$ can be generated as an $\cO$-module by $q$ elements. Fix $N > 2 C$, and let $\mathcal{Q}_N$ denote the set of finite places $v \not\in S$ of $F$ satisfying the following conditions:
    \begin{itemize}
        \item $q_v \equiv 1 \text{ mod }p^{N+C}$.
        \item The eigenvalues $\alpha_v, \beta_v$ of $\rho(\Frob_v)$ lie in $\cO$ and satisfy $\operatorname{ord}_\varpi(\alpha_v-\beta_v)^2 = C_4$.
    \end{itemize}
   We will show that the kernel of the map 
    \[ \operatorname{loc} : H^1(F_S / F, W_N^\ast) \to \prod_{v \in \mathcal{Q}_N} H^1_v(W_N^\ast) \]
    is annihilated by $\varpi^{3C(\rho)}$. Applying \cite[Lemma 6.5]{Kis04}, this will show that there is a subset $Q \subset \mathcal{Q}_N$ of order $q$ such that this kernel is equal to $H^1_{Q\text{-split}}(F_S / F, W_N^\ast)$; and this will prove the lemma. 

    Take $[\phi] \in \ker(\operatorname{loc})$, and suppose for the sake of contradiction that $\varpi^{3 C(\rho)} [\phi] \neq 0$. Let $f = \Res_{L_N / F}([\phi]) \in \Hom_{G_F}(G_{L_N}, W_N^\ast)$. By Lemma \ref{lem_annihilation_of_auxiliary_groups}(3) and the inflation-restriction exact sequence, $\varpi^{C(\rho)} f \neq 0$. By part (2) of the same Lemma, we have $\varpi^{N-1} W_N^\ast = W_{1}^\ast \leq \cO \cdot f(G_{L_N})$. 

    Choose any $\sigma \in G_{F(\zeta_{p^\infty})}$ such that $\rho(\sigma)$ has eigenvalues $\alpha, \beta \in \cO$ satisfying $\operatorname{ord}_\varpi(\alpha - \beta)^2 = C_4$. If $\phi(\sigma) \not\in (\sigma-1) W_N^\ast$, then we can choose, by the Chebotarev density theorem, a Taylor--Wiles place $v$ of level $N+C$ such that $\rho(\Frob_v) \equiv \rho(\sigma) \text{ mod }\varpi^N$ and  $\phi(\Frob_v) = \phi(\sigma)$, implying in particular that the image of $[\phi]$ in $H^1_v(W_N^\ast)$ is non-zero, a contradiction to $[\phi] \in \ker( \operatorname{loc})$. 

    Suppose instead that $\phi(\sigma) \in (\sigma-1) W_N^\ast$. Then we can choose $\tau \in G_{L_N}$ such that $f(\tau) \not\in (\sigma-1) W_N^\ast$. (Indeed, we can argue as in Lemma \ref{lem_TW_local_triviality}: there is a short exact sequence of $\cO[\sigma]$-modules
    \[ 0 \to M \to W_N^\ast \to \cO / (\varpi^N) \to 0, \]
    showing that $(\sigma-1)W_N^\ast$ is contained in a direct summand $\cO / (\varpi^N)$-submodule of $W_N^\ast$ of rank 2. In particular, $(\sigma-1) W_N^\ast$ does not contain $W_N^\ast[\varpi]$.) By the Chebotarev density theorem, there exists a Taylor--Wiles place $v$ of level $N+C$ such that $\rho(\Frob_v) \equiv \rho(\sigma) \equiv \rho(\tau \sigma) \text{ mod }\varpi^N$ and $\phi(\Frob_v) = \phi(\tau \sigma)$. Then $\phi(\Frob_v) \not\in (\Frob_v - 1) W_N^\ast$. This again contradicts $[\phi] \in \ker(\operatorname{loc})$, and completes the proof. 
\end{proof}
\begin{proposition}\label{prop_control_on_augmented_tangent_space}
    We can find an integer $q\geq 0$ with the following property: for any $N \geq 1$, there exists a Taylor--Wiles set $Q$ of level $N$ and with $|Q| = q$, and a morphism $\cO^q \to \mathfrak{p}_Q / \mathfrak{p}_Q^2$ with cokernel annihilated by $\varpi^{40 C(\rho)}$. 
\end{proposition}
\begin{proof}
   Define $q$ as in Lemma \ref{lem_killing_dual_Selmer}.  We are free to increase $N$, and in particular to assume that $N > 80 C(\rho)$. Combining Lemma \ref{lem_killing_dual_Selmer} and Proposition \ref{prop_application_of_Poitou--Tate} (applied with $C = 3 C(\rho)$), we see that for any $N > 80 C(\rho)$, we can find a Taylor--Wiles set $Q$ of level $N + 3 C(\rho)$ and with $q$ elements such that $H^1_{\mathcal{L}_{S \cup Q}}(W_N)$ is a $36 C(\rho)$-free $\cO / (\varpi^N)$-module of rank $q$. Combining this result with Lemma \ref{lem_comparison_of_tangent_spaces}, we find that $\Hom_\cO(\mathfrak{p}_Q / \mathfrak{p}_Q^2, \cO / \varpi^N)$ is a $40 C(\rho)$-free $\cO / (\varpi^N)$-module of rank $q$. The result follows.
\end{proof}
With these Galois-theoretic preliminaries out of the way, we are now ready to make the connection with automorphic forms. Let $D$ be a definite quaternion algebra over $F$, ramified precisely at $\Sigma$ and at the infinite places (this exists because $[F : \mathbb{Q}]$ and $|\Sigma|$ are both even). Fix a choice of maximal order $\cO_D \leq D$ and, for each finite place $v \not\in \Sigma$, an isomorphism $\cO_D \otimes_{\cO_F} \cO_{F_v} \cong M_2(\cO_{F_v})$. We use this isomorphism to identify $(D \otimes_F F_v)^\times$ with $\GL_2(F_v)$. If $U = \prod_v U_v \leq (D \otimes_F \mathbb{A}_F^\infty)$ is an open compact subgroup, then we write $H(U)$ for the set of functions $f : D^\times \mathbb{A}_F^{\infty, \times} \backslash (D \otimes_F \mathbb{A}_F^{\infty})^\times / U \to \cO$. This is a finite free $\cO$-module.

We next describe the particular level subgroups $U$ that we want to use. We define $U_0 = \prod_v U_{0, v}$, where $U_{0, v} = (\cO_D \otimes_{\cO_{F}} \cO_{F_v})^\times$. If $Q$ is a Taylor--Wiles set, then we define $U_0(Q) = \prod_v U_0(Q)_v$, where: 
\begin{itemize}
    \item If $v \not\in Q$, then $U_0(Q)_v = U_{0, v}$.
    \item If $v \in Q$, then $U_0(Q)_v = \Iw_v$. 
\end{itemize}
We define $U_1(Q) = \prod_v U_1(Q)_v$, where:
\begin{itemize}
    \item If $v \not\in Q$, then $U_1(Q)_v = U_{0, v}$. 
    \item If $v \in Q$, then 
    \[ U_1(Q)_v = \left\{ \left( \begin{array}{cc} a & b \\ c & d \end{array} \right) \in \Iw_v \mid a / d \text{ mod }\varpi_v \mapsto 1 \in k(v)^\times(p) \right\}, \]
\end{itemize}
where $k(v)^\times(p)$ denotes the maximal $p$-power quotient of $k(v)^\times$. We set $\Delta_Q = \prod_{v \in Q} k(v)^\times(p)$. Then $U_1(Q) \leq U_0(Q)$ is a normal subgroup, and there is a canonical isomorphism $U_0(Q) / U_1(Q) \cong \Delta_Q$. If $Q$ has level $N$, then we define $U_1(Q; N) \leq U_0(Q)$ to be the kernel of the induced surjective homomorphism $U_0(Q) \to \Delta_Q / (p^N)$. 

We next describe the Hecke operators we want to use. If $v \not\in \Sigma \cup Q$ is a finite place of $F$, then the usual unramified Hecke operators
\[ T_v = \left[ \GL_2(\cO_{F_v}) \left( \begin{array}{cc} \varpi_v & 0 \\ 0 & 1 \end{array}\right) \GL_2(\cO_{F_v}) \right], \]
\[ S_v = \left[ \GL_2(\cO_{F_v}) \left( \begin{array}{cc} \varpi_v & 0 \\ 0 & \varpi_v \end{array}\right) \GL_2(\cO_{F_v}) \right] \]
act on $H(U_0(Q))$ and $H(U_1(Q; N))$. If $v \in \Sigma$, then we write $U_v$ for the operator given by the double coset $[\Pi_v^\times \cO_{D_v}^\times]$ (where $\Pi_v \in \cO_{D_v}$ is a uniformizer). For the Hecke operators at places $v \in Q$, we use the homomorphisms $t_{v, i} : F_v^\times \to \mathcal{H}(\GL_2(F_v), U_1(Q; N)_v) \otimes_\mathbb{Z} \cO$ defined just before \cite[Proposition 2.2.9]{All23a}. We write $A_v = \cO[ t_{v, 1}( \varpi_v)^{\pm 1}, t_{v, 2}(\varpi_v)^{\pm 1}]$, a freely generated commutative subalgebra of the abstract Hecke algebra, and write $W_v = S_2$ for the group that acts on this by permuting $t_{v, 1}$ and $t_{v, 2}$. Then the algebra $A_v^{W_v}$ lies in the centre of the Hecke algebra. We set $A_Q = \otimes_{v \in Q} A_v$, $W_Q = \prod_{v \in Q} W_v$. The algebra $A_v$ preserves the subspace $H(U_0(Q)) \leq H(U_1(Q; N))$ and, for any $\alpha \in \cO_{F_v}^\times$, the action of $\alpha$ on $H(U_1(Q; N))$ via its projection to $\Delta_Q$ agrees with the action of $t_{v, 1}(\alpha) = t_{v, 2}(\alpha)^{-1}$. It is thus also natural to let $W_Q$ act on $\Delta_Q$ by letting the non-trivial element of $W_v$ act as inversion on $k(v)^\times$. We also define the element $\delta_Q = \prod_{v \in Q} (t_{v, 1}(\varpi_v) - t_{v, 2}(\varpi_v))^2 \in A_Q^{W_Q}$.  

We write $\mathbb{T}^{univ, S \cup Q}$ for the polynomial algebra over $\cO$ in the formal variables $T_v, S_v$ ($v \not\in S \cup Q$). We can identify $\mathbb{T}^{univ, S}$ with a subalgebra of $\mathbb{T}^{univ,S \cup Q} \otimes_\cO A_Q$ by letting the Hecke operators at $Q$ act through $A_Q^{W_Q}$. Thus $H(U_0(Q))$ and $H(U_1(Q; N))$ each have a natural structure of $\mathbb{T}^{univ,S \cup Q} \otimes_\cO A_Q[\Delta_Q]$-module, and hence of $\mathbb{T}^{univ,S}$-module. 

We write $\mathfrak{m} \leq \mathbb{T}^{univ, S}$ for the kernel of the homomorphism $\mathbb{T}^{univ, S} \to k$ that sends $X^2 - T_v X + q_v S_v$ $(v \not\in S)$ to $\det(X - \overline{\rho}(\Frob_v))$. We write $H(U_0(Q))^{ord} \leq H(U_0(Q))$ and $H(U_1(Q; N))^{ord} \leq H(U_1(Q; N))$ for the localisations where each operator $T_v$ ($v | p$) acts with eigenvalue residually equal to the trace of Frobenius on $\mathrm{WD}(\rho|_{G_{F_v}})$. 

Let $m_0 \geq 0$ denote the least common multiple of the exponents of the (abelian) Sylow $p$-subgroups of the groups $F^\times \backslash (U_0 \mathbb{A}_F^{\infty, \times} \cap t^{-1} D^\times t)$, as $t$ ranges over $(D \otimes_F \mathbb{A}_F^\infty)^\times$. Then $m_0$ is finite. The following proposition summarises the required properties of the $H(U_1(Q))$. 
\begin{proposition}\label{prop_properties_of_spaces_of_modular_forms}
    Let $N \geq 1$, and let $Q$ be a Taylor--Wiles datum of level $N + m_0$. Then:
    \begin{enumerate}
        \item The maximal ideal $\mathfrak{m}$ is in the support of $H(U_0)^{ord}$ and $H(U_0(Q))^{ord}$.
        \item $H(U_1(Q; N))_{\mathfrak{m}}^{ord}$ is a $\mathbb{T}^{univ, S \cup Q} \otimes_\cO A_Q[\Delta_Q / (p^N)]$-module, free as $\cO[\Delta_Q / (p^N)]$-module, and there is an isomorphism
        \[ H(U_1(Q; N))^{ord}_{\mathfrak{m}} \otimes_{\cO[\Delta_Q / (p^N)]} \cO \cong H(U_0(Q))^{ord}_{\mathfrak{m}} \]
        of $\mathbb{T}^{univ, S \cup Q} \otimes_\cO A_Q$-modules. 
        \item There exists a structure on $H(U_1(Q; N))^{ord}_{\mathfrak{m}}$ of $R_Q$-module such that if $\Lambda_1, \Lambda_2 : G_F \to R_Q$ are the coefficients of the universal characteristic polynomial, as defined in \cite[\S 1.10]{Che14}, then for any finite place $v\not\in S \cup Q$ of $F$, $\Lambda_1(\Frob_v)$ acts as $T_v$ and $\Lambda_2(\Frob_v)$ acts as $q_v S_v$; and for any $v \in Q$, $\sigma \in W_{F_v}$ with image $\Art_{F_v}(\alpha)$ in $W_{F_v}^{ab}$, $\Lambda_1(\sigma)$ acts as $t_{v, 1}(\alpha) + t_{v, 2}(\alpha)$, and $\Lambda_2(\sigma)$ acts as $t_{v, 1}(\alpha) t_{v, 2}(\alpha)$. In particular, the $\cO[\Delta_Q]^{W_Q}$-module structure on $H(U_1(Q; N))_{\mathfrak{m}}^{ord}$ induced by the homomorphism $\cO[\Delta_Q]^{W_Q} \to R_Q$ agrees with the tautological one. 
    \end{enumerate}
\end{proposition}
\begin{proof}
    We omit the details of the proof, which are very similar to the details of e.g. \cite[Propositions 2.3, 2.5]{New21}. 
\end{proof}
We recall that $\mathfrak{p}_\emptyset \leq R_\emptyset$ denotes the kernel of the homomorphism $R_\emptyset \to \cO$ associated to $\rho$. Let $\mathfrak{p}'_\emptyset \leq R_\emptyset$ denote the kernel of the homomorphism associated to $\rho'$. Thus, by assumption, we have an equality $(\mathfrak{p}_\emptyset, \varpi^{80 C(\rho)+1})=(\mathfrak{p}'_\emptyset, \varpi^{80 C(\rho)+1})$ of ideals of $R_\emptyset$.

We will patch these objects using ultrafilters. Using Proposition \ref{prop_control_on_augmented_tangent_space}, we fix an integer $q \geq 1$ and for each $N \geq 1$ a Taylor--Wiles datum $(Q_N, (\alpha_v, \beta_v)_{v \in Q_N})$ of level $N + m_0$ such that there is a morphism $\cO^q \to \mathfrak{p}_{Q_N} / \mathfrak{p}_{Q_N}^2$ with cokernel annihilated by $\varpi^{40 C(\rho)}$. Let $\mathcal{F}$ be a non-principal ultrafilter on $\mathbb{N}$, and let $\mathbf{R} = \prod_{N \geq 1} \cO$. If $I \in \mathcal{F}$, then we define $e_I \in \mathbf{R}$ by $e_{I, N} = 1$ if $N \in I$, $e_{I, N} = 0$ otherwise. Then $\{ e_I \mid I \in \mathcal{F} \}$ is a multiplicative subset of $\mathbf{R}$, and we write $\mathbf{R}_\mathcal{F}$ for the localisation at this multiplicative subset. 

We first define the patched pseudodeformation ring as
\[ R_\infty = \varprojlim_m R_\infty(m), \]
where 
\[ R_\infty(m) = \mathbf{R}_\mathcal{F} \otimes_{\mathbf{R}} \prod_{N \geq 1} R_{Q_N} / \mathfrak{m}_{R_{Q_N}}^m. \]
Writing $\mathbb{T}_{Q_N} = \mathbb{T}^{univ, S}(H(U_1(Q_N)^{ord}_{\mathfrak{m}})$ for the quotient of $\mathbb{T}^{univ, S}$ that acts faithfully on this module, we define a patched Hecke algebra as
\[ \mathbb{T}_\infty = \varprojlim_m \mathbb{T}_\infty(m), \]
where 
\[ \mathbb{T}_\infty(m) = \mathbf{R}_\mathcal{F} \otimes_{\mathbf{R}} \prod_{N \geq 1} \mathbb{T}_{Q_N} / (\mathfrak{m}^m). \]
The maps $R_{Q_N} \to \mathbb{T}_{Q_N}$ determine a map $R_\infty \to \mathbb{T}_\infty$. Similarly, there are canonical maps $R_\infty \to R_\emptyset$ and $\mathbb{T}_\infty \to \mathbb{T}_\emptyset$. 
\begin{lemma}\label{lem_finiteness_of_patched_Hecke_algebra}
    The rings $R_\infty$, $\mathbb{T}_\infty$ are complete Noetherian local $\cO$-algebras, and the maps $R_\infty \to \mathbb{T}_\infty$, $R_\infty \to R_\emptyset$, and $\mathbb{T}_\infty \to \mathbb{T}_\emptyset$ are surjective. 
\end{lemma}
\begin{proof}
    We can find an integer $g \geq 0$ such that, for any $N \geq 1$, $R_{Q_N}$ is a quotient of $\mathcal{O} \llbracket X_1, \dots, X_g \rrbracket$ (use the argument of \cite[Lemma 3.28]{Tho15}). This implies that $R_\infty$ is a complete Noetherian local ring. Since each map $R_{Q_N} \to \mathbb{T}_{Q_N}$ is surjective, we deduce the corresponding result for $\mathbb{T}_\infty$. The remaining claims follow from the definition.
\end{proof}
We write $\mathfrak{p}_\infty$, $\mathfrak{p}'_\infty \leq R_\infty$ for respective the pullbacks of the ideals $\mathfrak{p}_\emptyset$, $\mathfrak{p}'_\emptyset \leq R_\emptyset$. 
\begin{lemma}
    There exist homomorphisms of $\cO$-modules
    \[ \cO^q \to \mathfrak{p}_\infty / (\mathfrak{p}_\infty)^2, \]
    \[ \cO^q \to \mathfrak{p}'_\infty/ (\mathfrak{p}'_\infty)^2, \]
    each with cokernel annihilated by $\varpi^{40  C(\rho)}.$
\end{lemma}
\begin{proof}
    By construction, $\mathfrak{p}_\infty / (\mathfrak{p}_\infty)^2$ is an inverse limit of finite length $\cO$-modules $M_m$, each isomorphic to 
    \[ \frac{\mathfrak{p}_{Q_{N}} + \mathfrak{m}_{R_{Q_{N}}}^{m}}{ \mathfrak{p}_{Q_{N}}^2 + \mathfrak{m}_{R_{Q_{N}}}^{m}}  \]
    (with $N$ depending on $m$), with surjective transition maps, and a uniformly bounded number of generators. Thus for every $m \geq 1$, we can find a homomorphism $\cO^q \to M_m$ with cokernel annihilated by $\varpi^{40  C(\rho)}$. This property is preserved under passage to the inverse limit. 

    We want to deduce the corresponding result for $\mathfrak{p}'_\infty$. We first note that, for any $D \geq 1$, there is an isomorphism
    \[ (\mathfrak{p}'_\infty / (\mathfrak{p}'_\infty)^2) \otimes_\cO \cO / (\varpi^D) \cong (\mathfrak{p}'_\infty, \varpi^D) / ((\mathfrak{p}'_\infty)^2, \varpi^D). \]
    Indeed, it suffices to show that $((\mathfrak{p}'_\infty)^2, \varpi^D) \cap \mathfrak{p}'_\infty = ((\mathfrak{p}'_\infty)^2, \varpi^D \mathfrak{p}'_\infty)$ (equality as ideals of $R_\infty$). The right-hand side is clearly contained in the left-hand side. On the other hand, if $x \in (\mathfrak{p}'_\infty)^2$, $y \in R_\infty$, and $x + \varpi^D y \in \mathfrak{p}'_\infty$, then $y \in \mathfrak{p}'_\infty$ (since $\mathfrak{p}'_\infty$ is prime and $\varpi \not\in \mathfrak{p}'_\infty$), hence $x + \varpi^D y \in ((\mathfrak{p}'_\infty)^2, \varpi^D \mathfrak{p}'_\infty)$, so the left-hand side is contained in the right-hand side, and they're equal.

    Taking $D = 40 C(\rho) + 1$, we have $\mathfrak{p}_\infty + \varpi^D = \mathfrak{p}'_\infty + \varpi^D$, hence there is an isomorphism of $\cO$-modules
    \[ (\mathfrak{p}'_\infty / (\mathfrak{p}'_\infty)^2) \otimes_\cO \cO / (\varpi^{40 C(\rho) +1}) \cong (\mathfrak{p}_\infty / (\mathfrak{p}_\infty)^2) \otimes_\cO \cO / (\varpi^{40 C(\rho)+1}). \]
    We can now apply Lemma \ref{lem_existence_of_homomorphism} to conclude the existence of a homomorphism $\cO^q \to (\mathfrak{p}'_\infty / (\mathfrak{p}'_\infty)^2)$ with cokernel annihilated by $\varpi^{40  C(\rho)}$, as required. 
\end{proof}
\begin{proposition}
    The ideals $\mathfrak{p}'_\infty$, $\mathfrak{p}_\infty$ lie in the closed subset $\Spec \mathbb{T}_\infty \subset \Spec R_\infty$.
\end{proposition}
\begin{proof}
   We use a patching argument. Let $A = \cO[ \{ x_{i, j}^{\pm 1} \}_{i = 1, \dots, q}, {j = 1, 2}]$, $W = S_2^q$, $\delta_i = (x_{i, 1} - x_{i, 2})^2$ ($i = 1, \dots, q$), $\delta = \prod_{i=1}^q \delta_i$. Fixing for each $N \geq 1$ an ordering $\{ v_{N, 1}, \dots, v_{N, q} \}$ of $Q_N$, we can identify $A_{Q_N} = A$ and $W_{Q_N} = W$. Let $S_\infty = \cO \llbracket \mathbb{Z}_p^q \rrbracket$ and let $\mathfrak{a} \leq S_\infty$ denote the augmentation ideal. We let $W$ act on $\mathbb{Z}_p^q$ by inversion in each factor. Fixing for each $v \in Q_N$ a surjection $\mathbb{Z}_p \to k(v)^\times(p)$, we get a surjection $\mathbb{Z}_p^q \to \Delta_{Q_N}$, equivariant for the action of $W$.
   
   If $r \geq 1$, we define 
\[ H_r = \varprojlim_{m} H_r(m), \]
where
\[ H_r(m) = \mathbf{R}_\mathcal{F} \otimes_\mathbf{R} \prod_{N \geq m+r} \delta_{Q_N}^{2r} (H(U_1(Q_N; N))^{ord}_{\mathfrak{m}} / (\varpi^m, \mathfrak{a}^r)). \]
Then $H_r(m)$, and hence $H_r$, has a natural structure of $\mathbb{T}^{univ, S} \otimes_{\cO} A$-module, where $A$ acts on $H_r(m)$ via the diagonal map $A \to \prod_{N \geq 1} A_{Q_N}$. 
We observe that $H_r(m)$, and hence $H_r$, has a natural structure of $\mathbb{T}_\infty$-module. Indeed, it suffices to show that for fixed $m, r \geq 1$, and for any $N \geq m+r$, there is an integer $s(m, r)$ such that the action of $\mathbb{T}_{Q_N}$ on $\delta_{Q_N}^{2r} (H(U_1(Q_N; N))^{ord}_{\mathfrak{m}} / (\varpi^m, \mathfrak{a}^r))$ factors through $\mathbb{T}_{Q_N} / (\mathfrak{m}^{s(m,r)})$. This is true because these $\cO$-modules may be generated by a number of elements depending only on $r$ (cf. \cite[Lemma 4.11]{Aca26}). 

We claim that the localisation $H_{r, \mathfrak{p}'_\infty}$ is a non-zero finite free $S_{\infty, \mathfrak{a}} / (\mathfrak{a}^r)$-module. Indeed, let
\[ H'_r = \varprojlim_{m} \left( \mathbf{R}_\mathcal{F} \otimes_\mathbf{R} \prod_{N \geq m+r}  H(U_1(Q_N; N))^{ord}_{\mathfrak{m}} / (\varpi^m, \mathfrak{a}^r) \right). \]
If $N \geq m+r$ then $S_\infty / (\varpi^m, \mathfrak{a}^r)$ is a quotient of $\cO[ \Delta_{Q_N} ]$, so  
\[ H(U_1(Q_N; N))^{ord}_{\mathfrak{m}} / (\varpi^m, \mathfrak{a}^r) \]
is a flat module over $S_\infty / (\varpi^m, \mathfrak{a}^r)$. The proof of \cite[Lemma 4.9]{New23} applies without change to show that $H'_r$ is a flat $S_{\infty} / (\mathfrak{a}^r)$-module. There are maps $H_r \to H'_r$, $H'_r \to H_r$ of $\mathbb{T}^{ univ, S} \otimes_\cO S_\infty$-modules whose composite in each direction is equal to multiplication by $\delta^{2r}$. Therefore $H_r[\delta^{-1}]$ is a flat $S_\infty / (\mathfrak{a}^r)$-module. 

On the other hand, the action of each $\delta_i$ on $H_r$ agrees with the action of an element of $R_\infty$ which is not contained in $\mathfrak{p}'_\infty$: more precisely, it is the element of $R_\infty$ whose image in $R_\infty(m)$ is 
\[ \left( \operatorname{disc} \det( X - \Frob_{v_{N, i}} ) \right)_{N \geq 1}, \]
and whose image in $R_\infty(m) / (\mathfrak{p}'_\infty) \cong \cO / (\varpi^m)$ is congruent mod $(\varpi^m, \varpi^{80 C(\rho)+1})$ to $(\alpha_{v_{N, i}} - \beta_{v_{N, i}})^2$ (for some $N$ depending on $m$); in particular, non-zero as soon as $m > C_4$. It follows that $\delta = \delta_1 \dots \delta_q$ agrees with the action of an element of $R_\infty$ not contained in the prime ideal $\mathfrak{p}'_\infty$, and therefore that $H_{r, \mathfrak{p}'_\infty}$ is a localisation of $H_r[\delta^{-1}]$, and so is flat over $S_{\infty, \mathfrak{a}} / (\mathfrak{a}^r)$. 

Then, using \cite[Proposition 3.1]{New23}, we see that there are maps 
\[ \alpha_r : H(U_0)^{ord}_{\mathfrak{m}} \otimes_{A^W} A \to H'_r / (\mathfrak{a}), \]
\[ \beta_r : H'_r / (\mathfrak{a}) \to H(U_0)^{ord}_{\mathfrak{m}} \otimes_{A^W} A, \]
of $\mathbb{T}^{univ, S}$-modules, whose composite in each direction is given now by multiplication by $\delta^2$. This shows that $H_{r, \mathfrak{p}'_\infty}$ is non-zero and finitely generated as $S_{\infty, \mathfrak{a}}$-module, since $\beta_r$ induces  an isomorphism
\[ H_{r, \mathfrak{p}'_\infty} / (\mathfrak{a}) \cong \left( H(U_0)^{ord}_{\mathfrak{m}} \otimes_{A^W} A \right)_{\mathfrak{p}'_\infty} = \left( H(U_0)^{ord}_{\mathfrak{m}} \otimes_{A^W} A \right)_{\mathfrak{p}'_\emptyset}, \]
and the right-hand module is non-zero by the existence of $\pi$ and the Jacquet--Langlands correspondence. 

We now define
\[ \widehat{H}_\infty = \varprojlim_r H_{r, \mathfrak{p}'_\infty}. \]
This is a finite free module over the completion $\widehat{S}_\infty$ of $S_{\infty, \mathfrak{a}}$, that moreover receives an action of the completion $\widehat{R}_\infty$ of $R_{\infty, \mathfrak{p}'_\infty}$ via the completion $\widehat{\mathbb{T}}_\infty$ of $\mathbb{T}_{\infty, \mathfrak{p}'_\infty}$. By construction, there is an isomorphism $\widehat{H}_\infty / (\mathfrak{a}) \cong (H(U_0)^{ord}_{\mathfrak{m}} \otimes_{A^W} A)_{\mathfrak{p}'_\emptyset}$. 

The action of $\widehat{S}_\infty^W$ on $\widehat{H}_\infty$ factors through the action of $\widehat{R}_\infty$. Since $\widehat{S}_\infty^W$ is a regular local ring of dimension $q$, it follows that $\widehat{H}_\infty$ is a non-zero finite free $\widehat{S}_\infty^W$-module, hence a Cohen--Macaulay $\widehat{R}_\infty$-module, and therefore that each irreducible component of the support of $\widehat{H}_\infty$ in $\Spec \widehat{R}_\infty$ has dimension $q$. In particular, $\dim \widehat{R}_{\infty, \mathfrak{p}'_\infty} \geq q$, hence $\dim R_{\infty, \mathfrak{p}'_\infty} \geq q$. 

We may now apply Theorem \ref{thm_application_of_Hensel}, with $R = R_\infty$, $P_1 = \mathfrak{p}'_\infty$, and $P_2 = \mathfrak{p}_\infty$, to conclude that $\Spec R_\infty$ has a unique irreducible component containing $\mathfrak{p}'_\infty$, and that this irreducible component contains $\mathfrak{p}_\infty$. To prove the proposition, it is therefore enough to show that the map $R_\infty \to \mathbb{T}_\infty$ induces an isomorphism after passage to complete local rings at $\mathfrak{p}'_\infty$: or in other words, that the surjective homomorphism $\widehat{R}_\infty \to \widehat{\mathbb{T}}_\infty$ is an isomorphism. This is true: $\widehat{R}_\infty$ is regular, and the Auslander--Buchsbaum formula implies that $\widehat{H}_\infty$ is a non-zero free $\widehat{R}_\infty$-module. Since the action of $\widehat{R}_\infty$ on $\widehat{H}_\infty$ factors through the map to $\widehat{\mathbb{T}}_\infty$, this map must be an isomorphism. 
\end{proof}
To finish the proof, we need to show that $\mathfrak{p}_\infty \leq \mathbb{T}_\infty$ is pulled back from $\mathbb{T}_\emptyset$. However, it seems difficult to show this directly without a better understanding of the space of modular forms as a module for the Hecke algebra. We therefore now introduce spaces of Hilbert modular forms, for which multiplicity 1 can be established using the $q$-expansion principle, in order to conclude the proof. 

Let $\mathfrak{n} \leq \cO_F$ be a non-zero ideal satisfying the following conditions:
\begin{itemize}
    \item $\mathfrak{n}$ is prime to $S$.
    \item For any quadratic CM extension $K / F$ such that either $K = F(\zeta)$ for some root of unity $\zeta$ or $K = F(\sqrt{\beta})$ for some $\beta \in \cO_F^\times$, and for any prime $r$ such that $\zeta_r \in K$, there exists a non-zero prime ideal $\mathfrak{q}$ of $\cO_F$ dividing $\mathfrak{n}$ which is prime to $r$ and inert in $K$.
\end{itemize}
(There are infinitely many such ideals, and the choice is unimportant: $\mathfrak{n}$ is used only to rigidify the moduli problem underlying the definition of the spaces of Hilbert modular forms.) 
Let $V_0 = \prod_v V_{0, v} \leq \prod_v \GL_2(\cO_{F_v}) \leq \GL_2(\mathbb{A}_F^\infty)$ be the open compact subgroup defined as follows:
\begin{itemize}
    \item If $v \nmid \mathfrak{n}$ and $v \not\in \Sigma$, then $V_{0, v}= \GL_2(\cO_{F_v})$.
    \item If $v | \mathfrak{n}$, then
    \[ V_{0, v} = \left\{ \left( \begin{array}{cc} a & b \\ c & d \end{array} \right) \in \Iw_v \mid d \equiv 1 \text{ mod } \mathfrak{n} \cO_{F_v} \right \}. \]
    \item If $v \in \Sigma$, then $V_{0, v} = \Iw_v$.
\end{itemize}
Then $V_0$ is sufficiently small, in the sense of \cite{Dia24}.

If $Q$ is a Taylor--Wiles set, then we define further open compact subgroups $V_1(Q) \leq V_0(Q) \leq V_0$ as follows:
\begin{itemize}
    \item If $v \not\in Q$, then $V_1(Q)_v = V_0(Q)_v = V_{0, v}$.
    \item If $v \in Q$, then $V_1(Q)_v = U_1(Q)_v$ and $V_0(Q)_v = U_0(Q)_v$.
\end{itemize}
We define spaces of Hilbert modular forms following \cite{Dia24}. For an open compact subgroup $V = \prod_v V_v \leq V_0$ with $V_p = V_{0, p}$, and $R = \cO$ or $\cO / (\varpi^M)$ for some $M \geq 1$, let $S(V, R)$ denote the space defined as $S_{\overset{\rightarrow}{k}, \overset{\rightarrow}{m}}(V, R)$ on \cite[p. 285]{Dia24}, with 
\[ \overset{\rightarrow}{k} = (2, 2, \dots, 2),\,\, \overset{\rightarrow}{m} = 0. \]
By definition, it is a subspace of the space of sections of a certain coherent sheaf on the base change to $R$ of the minimal compactification of the Pappas--Rapoport model of the Hilbert modular variety. This model depends on an auxiliary choice (of ordering of embeddings of $F_v$ in $E$, for each $v | p$) but this choice is immaterial for the arguments given here. 
\begin{proposition}\label{prop_basic_properties_of_HMF}
    \begin{enumerate}
        \item For any $V$ as above, $S(V, \cO)$ is a finite free $\cO$-module, and there is a canonical isomorphism of $S(V, \cO) \otimes_{\cO, \iota} \mathbb{C}$ with the space of cuspidal automorphic forms $f : \GL_2(\mathbb{A}_F) \to \mathbb{C}$ satisfying the following conditions:
        \begin{enumerate}
            \item For all $u \in U$, $u_\infty \in \mathrm{SO}_2(F \otimes_{\mathbb{Q}} \mathbb{R})$, $g \in \GL_2(\mathbb{A}_F)$, $f(g u u_\infty) = \det(u_\infty) j(u_\infty, i)^{-2} f(g)$.
            \item For all $g^\infty \in \GL_2(\mathbb{A}_F^\infty)$, the function $f_{g^\infty} : \mathfrak{h}^{\Hom(F, \mathbb{C})} \to \mathbb{C}$ defined by $f_{g^\infty}(g_\infty i) = \det(g_\infty)^{-1} j(g_\infty, i)^2 f(g^\infty g_\infty)$ ($g_\infty \in \GL_2(F_\infty)^\circ$) is holomorphic. 
        \end{enumerate}
        \item The above isomorphism is compatible with the action of Hecke operators, in the following sense: for every finite place $v$ of $F$, the action of the Hecke algebra $\mathcal{H}(\GL_2(F_v), V_v)$ on the right-hand side preserves $S(V, \cO)$. 
        \item If $V' = \prod_v V'_v$ is a normal, open compact subgroup of $V$ such that $V'_p = V_p = V_{0, p}$, then $S(V, R) = S(V', R)^{V / V'}$.
    \end{enumerate}
\end{proposition}
\begin{proof}
    For the first part, see \cite[\S 3.5]{Dia23}. We omit the (standard) definition of the factor of automorphy $j(g, \tau)$. The second follows from the construction of Hecke operators on integral Hilbert modular forms in e.g.\ \cite{Dia24}, together with the verification (by computation of action on $q$-expansions) that they agree with the classical ones after scalar extension to the complex numbers. The third part follows from the Koecher principle, since $V / V'$ acts freely on the open Pappas--Rapoport model of the Hilbert modular variety. 
\end{proof}
In particular, if $Q$ is a Taylor--Wiles set, then we have $S(V_1(Q), R)^{\Delta_Q} = S(V_0(Q), R)$. 

We now define augmented Hecke algebras. If $Q$ is a Taylor--Wiles set, let $\widetilde{\mathbb{T}}^{univ}_Q$ denote the polynomial algebra over $\cO$ in the variables $T_v, S_v$ ($v \not\in \Sigma \cup Q$ and $v \nmid \mathfrak{n}$) and $U_v, S_v$ ($v \in \Sigma \cup Q$ or $v | \mathfrak{n}$). Then $S(V_0(Q),R)$ and $S(V_1(Q),R)$ are naturally $\widetilde{\mathbb{T}}^{univ}_Q$-modules and the inclusion $S(V_0(Q)) \to S(V_1(Q))$ is compatible with this action. We note that if $v \in Q$ then $U_v, S_v$ generate, with their inverses, the subalgebra $A_v$ of the abstract local Hecke algebra. Accordingly, we can identify $\widetilde{\mathbb{T}}^{univ}_Q$ with a subalgebra of $\widetilde{\mathbb{T}}^{univ}_\emptyset \otimes_{A_Q^{W_Q}} A_Q$.  

We define $\widetilde{\mathfrak{m}} \leq \widetilde{\mathbb{T}}^{univ}_\emptyset$ to be any maximal ideal in the support of $\iota^{-1}(\pi^\infty)^{V_0}$. Then $\widetilde{\mathfrak{m}}$ has residue field $k$ (because of our assumption at the start of this section that $E$ is sufficiently large), and $\mathfrak{m} \leq \widetilde{\mathfrak{m}}$. (The only ambiguity in choosing $\widetilde{\mathfrak{m}}$ is which $U_v$-eigenvalue to take at each place $v | \mathfrak{n}$.) We define $\widetilde{\mathfrak{m}}_Q = (\widetilde{\mathfrak{m}}, \mathfrak{a}, \{ t_{v, 1}(\varpi_v) - \alpha_v, t_{v, 2}(\varpi_v) - \beta_v \}_{v \in Q}) \leq \widetilde{\mathbb{T}}^{univ}_Q[\Delta_Q]$; then $\widetilde{\mathfrak{m}}_Q$ is another maximal ideal with residue field $k$, which now occurs in the support of $(\iota^{-1} \pi^\infty)^{V_0(Q)}$.  We write $\widetilde{\mathbb{T}}_Q$ for the quotient of $\widetilde{\mathbb{T}}^{univ}_Q[\Delta_Q]$ that acts faithfully on $S(V_1(Q), \cO)_{\widetilde{\mathfrak{m}}_Q}$.
\begin{proposition}\label{prop_multiplicity_1}
    If $Q$ is a Taylor--Wiles set, then $\Hom_\cO(S(V_1(Q), \cO)_{\widetilde{\mathfrak{m}}_Q}, \cO)$ is a free $\widetilde{\mathbb{T}}_Q$-module of rank $1$. 
\end{proposition}
\begin{proof}
    It's a faithful module, so we just need to show that it's cyclic, or equivalently that $\dim_k (S(V_1(Q), \cO) \otimes_\cO k)[\widetilde{\mathfrak{m}}_Q] = 1$. Since $S(V_1(Q), \cO) \otimes_\cO k \leq S(V_1(Q), k)$, it suffices to show the stronger statement that $\dim_k S(V_1(Q), k)[\widetilde{\mathfrak{m}}_Q] = 1$. By Proposition \ref{prop_basic_properties_of_HMF}(3), it even suffices to show that $\dim_k S(V_0(Q), k)[\widetilde{\mathfrak{m}}_Q] = 1$. This appears to be a standard fact, as in the  classical case $F = \mathbb{Q}$, but for want of a suitable reference we give a proof using the formulae for the action of Hecke operators given on $q$-expansions given in \cite{Dia24}.

    Let $\mathfrak{d} = \mathfrak{d}_{F / \mathbb{Q}}$ be the different. If $t \in (\mathbb{A}_F^{p, \infty})^\times$, then we define $J_t = F \cap t^{-1} \widehat{\cO}_F$. According to \cite[\S 6.4]{Dia24}, we can define, for any $t \in (\mathbb{A}_F^{p, \infty})^\times$, $m \in (\mathfrak{d}^{-1} J_t)_+$, and $f \in S(V_0(Q), k)$, a $q$-expansion coefficient $r^t_m(f)$. These coefficients have the property that if $t_1, \dots, t_h$ is a set of representatives for $\cO_{F, (p), +} \backslash (\mathbb{A}_F^{p, \infty})^\times / \widehat{\cO}_F^{p, \times}$, then the map
    \[ S(V_0(Q), k) \to \oplus_{i=1}^{t_i} \left\{ \sum_{m \in (\mathfrak{d}^{-1} J_{t_i})_+} r_m^{t_i} q^m \mid r_m^{t_i} \in k \right\}, \]
    \[ f \mapsto \left( \sum r_m^{t_i}(f) q^m \right)_{i=1}^h, \]
    is injective. If $t' = \alpha t u$ represent the same coset (equivalently, $J_t, J_{t'}$ define the same narrow ideal class), then $J_{t'} = \alpha^{-1} J_t$ and we have the formula 
    \begin{equation}\label{eqn_q_exp_formula}
    r_m^{t'}(f) = r_{\alpha m}^t(f)
    \end{equation} for all $m \in (\mathfrak{d}^{-1} J_{t'})_+$. We can phrase this slightly differently as follows: let $\mathfrak{b} \leq \mathfrak{d}^{-1}$ be a fractional ideal, and suppose that $\mathfrak{b} = m J_t^{-1}$ for $m$, $t$, as above. Write $c_{\mathfrak{b}}(f) = r^t_m(f)$. Then $c_{\mathfrak{b}}(f)$ is well-defined, and $f$ is determined by the coefficients $c_{\mathfrak{b}}$ as $\mathfrak{b}$ ranges over the set of all fractional ideals contained in $\mathfrak{d}^{-1}$.

    The effect of Hecke operators on $q$-expansions is given as follows: if $v \not\in \Sigma \cup Q$ and $v \nmid \mathfrak{n}$, then
    \[ c_{\mathfrak{b}}(T_v f) = c_{v \mathfrak{b}}(f) + q_v c_{v^{-1} \mathfrak{b}}(S_v f), \]
    where by convention $c_{v^{-1} \mathfrak{b}} = 0$ if $v^{-1} \mathfrak{b} \not\subset \mathfrak{d}^{-1}$. If $v \in \Sigma \cup Q$ or $v | \mathfrak{n}$, then
    \[ c_{\mathfrak{b}}(U_v f) = c_{v \mathfrak{b}}(f). 
    \]
    (These formulae are easily deduced from the ones for the coefficients $r_m^t$ proved in \cite[\S 6]{Dia24}.) 

    Suppose then that $f \in S(V_0(Q), k)[\widetilde{\mathfrak{m}}_Q]$. Using the above formulae, one easily checks by induction on the index $[ \mathfrak{d}^{-1} : \mathfrak{b}]$ that $f$ is determined by the coefficient $c_{\mathfrak{d}^{-1}}(f)$. In particular, $S(V_0(Q), k)[\widetilde{\mathfrak{m}}_Q]$ has dimension at most 1.  
\end{proof}
Let $W_0$, $W_0(Q), W_1(Q; N) \leq (D \otimes_F \mathbb{A}_F^\infty)^\times$ denote the respective subgroups of $U_0$, $U_0(Q)$, $U_1(Q; N)$ whose local components at places $v | \mathfrak{n}$ lie in $V_{0, v}$. 

We can now complete the proof of Theorem \ref{thm_modularity_by_close_approximation}.
\begin{proof}[Proof of Theorem \ref{thm_modularity_by_close_approximation}]
Let $\widetilde{\mathbb{T}}_{Q_N}'$ be the quotient of $\widetilde{\mathbb{T}}^{univ}_{Q_N}$ that acts faithfully on $H(W_1(Q_N; N))_{\widetilde{\mathfrak{m}}_{Q_N}}$, and let $\mathbb{T}'_{Q_N}$ be the image of $\mathbb{T}^{univ, S}$ in $\widetilde{\mathbb{T}}_{Q_N}'$.
Define 
\[ \mathbb{T}_\infty' = \varprojlim_m \mathbf{R}_{\mathcal{F}} \otimes_{\mathbf{R}} \prod_{N \geq 1} \mathbb{T}'_{Q_N} / (\mathfrak{m}^m), \]
\[ \widetilde{\mathbb{T}}_\infty' = \varprojlim_m \mathbf{R}_{\mathcal{F}} \otimes_{\mathbf{R}} \prod_{N \geq 1} \widetilde{\mathbb{T}}'_{Q_N} / (\mathfrak{m}^m), \]
and 
\[ \widetilde{\mathbb{T}}_\infty = \varprojlim_m  \mathbf{R}_{\mathcal{F}} \otimes_{\mathbf{R}} \prod_{N \geq 1} \widetilde{\mathbb{T}}_{Q_N} / (\mathfrak{m}^m). \]
Then $\mathbb{T}'_\infty$ is a complete Noetherian local $\cO$-algebra equipped with a surjective map $\mathbb{T}'_\infty \to \mathbb{T}_\infty$ (by the same argument as in the proof of Lemma \ref{lem_finiteness_of_patched_Hecke_algebra}). The rings $\widetilde{\mathbb{T}}'_\infty$ and $\widetilde{\mathbb{T}}_\infty$ are also complete Noetherian local $\cO$-algebras. Indeed, it suffices to show that e.g.\ $\widetilde{\mathbb{T}}'_{Q_N}$ is a quotient of $\cO \llbracket X_1, \dots, X_g \rrbracket$, for some $g$ independent of $N$. This will follow if we can show that $\widetilde{\mathbb{T}}'_{Q_N}$ is a finite $\mathbb{T}'_{Q_N}$-algebra, with number of generators bounded independently of $N$. As an algebra, it is generated by the Hecke operators $U_v$ ($v \in \Sigma \cup Q$ or $v | \mathfrak{n}$) and $T_v$ ($v | p$) and the elements of $\Delta_{Q_N}$. The $U_v$ operators satisfy quadratic polynomials over the Bernstein centre. For each $v | p$, the `unit root' $\alpha_v$ of $X^2 - T_v X + q_v S_v$ satisfies the quadratic polynomial $\det(X - \rho(\Frob_v))$ (where $\Frob_v$ is an arbitrary choice of Frobenius lift), and $T_v = \alpha_v + q_v \alpha_v^{-1}$, so $T_v$ is also integral over the anaemic Hecke algebra. The image of $S_\infty^{W}$ in $\widetilde{\mathbb{T}}'_{Q_N}$ is contained in $\mathbb{T}'_{Q_N}$. The finiteness follows. An identical argument applies to $\widetilde{\mathbb{T}}_{Q_N}$.

Since the homomorphism $\mathbb{T}'_{Q_N} \to \widetilde{\mathbb{T}}'_{Q_N}$ is injective, the Fitting ideal 
\[\operatorname{Fitt}_{\mathbb{T}'_{Q_N} / (\mathfrak{m}^m)} \widetilde{\mathbb{T}}'_{Q_N} / (\mathfrak{m}^m) \]
is zero. By passage to the limit, it follows that $\operatorname{Fitt}_{\mathbb{T}'_\infty} \widetilde{\mathbb{T}}'_\infty = 0$. On the other hand, $\widetilde{\mathbb{T}}'_\infty$ is a finite $\mathbb{T}'_\infty$-algebra, so it follows that we can extend the pullback of $\mathfrak{p}_\infty$ to $\mathbb{T}_\infty'$ to a prime ideal $\widetilde{\mathfrak{p}}_\infty \leq \widetilde{\mathbb{T}}'_\infty$, kernel of a homomorphism $\widetilde{\mathbb{T}}'_\infty \to \cO$ (using here our hypothesis that $\rho'$ is defined over $\cO_1$ and $\cO$ contains the roots of every monic quadratic polynomial in $\cO_1[X]$). 

Fix $M \geq 1$, and consider the homomorphism $\widetilde{\mathbb{T}}'_\infty \to \cO / (\varpi^M)$ obtained by reduction modulo $\varpi^M$. This determines, for some $N \geq M$, a homomorphism $\pi_M : \widetilde{\mathbb{T}}'_{Q_N} \to \cO / (\varpi^M)$. Let $A = H(W_1(Q_N; N))_{\widetilde{\mathfrak{m}}_{Q_N}}$, let $A^{nt} \leq A$ denote the $\cO$-submodule spanned by functions $f : (D \otimes_F \mathbb{A}_F^\infty)^\times \to \cO$ that factor through the reduced norm, and let $A^{cusp} = A / A^{nt}$. Then $A^{nt}$ is a $\widetilde{\mathbb{T}}'_{Q_N}$-submodule; let $\widetilde{\mathbb{T}}''_{Q_N}$ denote the quotient of $\widetilde{\mathbb{T}}'_{Q_N}$ that acts faithfully on $A^{cusp}$, and define
\[ I_N = \ker( \widetilde{\mathbb{T}}'_{Q_N} \to \widetilde{\mathbb{T}}''_{Q_N}). \]
Then $\widetilde{\mathbb{T}}''_{Q_N}$ is naturally a quotient of $\widetilde{\mathbb{T}}_{Q_N}$, by the Jacquet--Langlands correspondence. On the other hand, choosing an unramified place $v \not\in S \cup Q_N$ such that $q_v \equiv 1 \text{ mod }p^M$ and $\ord_\varpi \disc \det(X - \rho(\Frob_v)) = C_4$, we have $(T_v^2 - (q_v+1)^2 S_v) I_N = 0$ (because $(T_v^2 - (q_v+1)^2 S_v)$ annihilates $A^{nt}$) and $(\pi_M(T_v^2 - (q_v+1)^2 S_v)) = (\varpi^{C_4})$, hence $\varpi^{C_4} \pi_M(I_N) = 0$. We deduce that $\pi_M \text{ mod }\varpi^{M - C_4}$ factors through a homomorphism $\widetilde{\mathbb{T}}''_{Q_N} \to \cO / (\varpi^{M - C_4})$, and therefore lifts to a homomorphism $\widetilde{\mathbb{T}}_{Q_N} \to \cO / (\varpi^{M - C_4})$.

 By Proposition \ref{prop_multiplicity_1}, $(S(V_1(Q), \cO) / (\varpi^{M-C_4})) [ \widetilde{\mathfrak{p}}_\infty ]$ is a free $\cO / (\varpi^{M-C_4})$-module of rank 1. By construction, $\mathfrak{p}_\infty$ contains $\mathfrak{a}^W$; since $\widetilde{\mathfrak{p}}_\infty$ is prime, it must contain $(\mathfrak{a})$, so we find that 
\[ (S(V_1(Q_N), \cO) / (\varpi^{M-C_4})) [ \widetilde{\mathfrak{p}}_\infty ] = (S(V_0(Q_N), \cO) / (\varpi^{M-C_4})) [ \widetilde{\mathfrak{p}}_\infty ], \]
and so the right-hand module is also free of rank 1 over $\cO / (\varpi^{M-C_4})$, and we obtain a homomorphism $\widetilde{\mathbb{T}}^{univ}_{Q_N}(S(V_0(Q_N), \cO))_{\widetilde{\mathfrak{m}}_{Q_N}} \to \cO / (\varpi^{M-C_4})$. On the other hand, the inclusion
\[ S(V_0, \cO)^{ \mathbb{F}_2^q} \to S(V_0(Q_N), \cO),  \]
\[ (f_{\underline{x}})_{{\underline{x}} \in \mathbb{F}_2^q} \mapsto \sum_{\underline{x}} \left( \prod_{j=1 : x_j = 1}^q \left( \begin{array}{cc} \varpi_{v_j} & 0 \\ 0 & 1 \end{array}\right) \right) f_{\underline{x}} \]
is injective, and determines a surjective homomorphism
\[ \widetilde{\mathbb{T}}^{univ}_{Q_N}(S(V_0(Q_N),\cO)) \to \widetilde{\mathbb{T}}^{univ}_{Q_N}(S(V_0, \cO)^{\mathbb{F}_2^q}). \]
Writing $J_N$ for the kernel of this surjective homomorphism, we see that we have 
\[ \left(\prod_{v \in Q_N} (t_{v, 1}(\varpi_v) - q_v t_{v, 2}(\varpi_v))(t_{v, 2}(\varpi_v) - q_v t_{v, 1}(\varpi_v)) \right) J_N = 0 \]
in $\widetilde{\mathbb{T}}^{univ}_{Q_N}(S(V_0(Q_N), \cO))$ (because 
\[ \left(\prod_{v \in Q_N} (t_{v, 1}(\varpi_v) - q_v t_{v, 2}(\varpi_v))(t_{v, 2}(\varpi_v) - q_v t_{v, 1}(\varpi_v)) \right) \]
annihilates the Iwahori invariants of the Steinberg representation). On the other hand, the image of the displayed element in $\cO / (\varpi^{M-C_4})$ divides $\varpi^{q C_4}$. Therefore, the  homomorphism 
\[ \widetilde{\mathbb{T}}^{univ}_{Q_N}(S(V_0(Q_N))) \to \cO / (\varpi^{M - C_4}) \]
determines a homomorphism 
\[ \widetilde{\mathbb{T}}^{univ}_{Q_N}(S(V_0)^{\mathbb{F}_2^q}) \to \cO / (\varpi^{M-(q+1)C_4}), \]
and hence a homomorphism 
\[ \mathbb{T}^{univ, S}( S(V_0, \cO) ) \to \cO / (\varpi^{M-(q+1)C_4}) \]
(simply by restriction to the subalgebra $\mathbb{T}^{univ, S}( S(V_0, \cO)^{\mathbb{F}_2^q} ) \cong \mathbb{T}^{univ, S}( S(V_0, \cO) )$). Now letting $M \to \infty$ gives the result. 
\end{proof}

    \section{Deduction of Theorem \ref{introthm_potmod}}

    In this section, we deduce our main theorem from the introduction, Theorem \ref{introthm_potmod}. For the convenience of the reader, we restate it here:
    \begin{theorem}
    Let $F$ be a totally real number field, let $p$ be a prime, and let $\rho : G_F \to \GL_2(\overline{\mathbb{Q}}_p)$ be a continuous, irreducible representation satisfying the following conditions:
    \begin{enumerate}
        \item $\rho$ is unramified at all but finitely many places. 
        \item For each place $v | p $ of $F$, $\rho|_{G_{F_v}}$ is potentially crystalline and ordinary of Hodge--Tate weights $\{0, 1 \}$. 
        \item $\det \rho$ is totally odd. 
    \end{enumerate}
     Then $\rho$ is potentially modular, in the sense that there is a finite totally real extension $F' / F$, an isomorphism $\iota : \overline{\mathbb{Q}}_p \to \mathbb{C}$, and a cuspidal, regular algebraic automorphic representation $\pi$ of $\GL_2(\mathbb{A}_{F'})$ such that $\rho|_{G_{F'}} \cong r_{\pi, \iota}$. 
\end{theorem}
\begin{proof}
    Consider the Zariski closure $G \leq \GL_2$ of $\rho(G_F)$. Then $G$ is a reductive group, whose image in $\PGL_2$ is either $\PGL_2$, or the normaliser of a rank 1 torus. In the latter case, $\rho$ is induced from a quadratic extension, so one can show that $\rho$ is modular by automorphic induction. Let us therefore suppose that the projective image of $G$ is $\PGL_2$. After replacing $\rho$ by a conjugate, we can assume that it takes values in $\cO_0$, the ring of integers in a coefficient field $E_0 / \mathbb{Q}_p$ that contains the image of every embedding $F \hookrightarrow \overline{\mathbb{Q}}_p$. Let $E_1 / E_0$ be the compositum of all quadratic extensions of $E_0$ in $\overline{\mathbb{Q}}_p$, and let $E / E_1$ be the compositum of all quadratic extensions of $E_1$. Let $\cO_1 \leq E_1$ and $\cO \leq E$ be the respective rings of integers. 
    
    We can choose a finite, totally real extension $F_0 / F$ with the following properties:
    \begin{itemize}
        \item $(\rho \times \epsilon)(G_F) = (\rho \times \epsilon)(G_{F_0})$.
        \item For each finite place $v$ of $F_0$, $\mathrm{WD}(\rho|_{G_{F_{0, v}}})$ is semistable. 
        \item Writing $\Sigma_0$ for the set of places of $F_0$ such that $\mathrm{WD}(\rho|_{G_{F_{0, v}}})$ is ramified, we have that $[F_0 : \mathbb{Q}]$ and $|\Sigma_0|$ are both even. 
    \end{itemize}
    This is possible by (the proof of) \cite[Lemma 5.1]{New23}. Fix an isomorphism $\iota : \overline{\mathbb{Q}}_p \to \mathbb{C}$, and let $S_0$ denote the union of $\Sigma_0$ with the set of $p$-adic places of $F_0$. By Theorem \ref{thm_modularity_by_close_approximation}, we can find a constant $C \geq 1$ with the following property:
    \begin{itemize}
        \item Let $F_1 / F_0$ be an $S_0$-split totally real extension such that $(\rho \times \epsilon)(G_{F_1}) = (\rho \times \epsilon)(G_{F_0})$, and suppose there exists a continuous representation $\rho' : G_{F_1} \to \GL_2(\cO_1)$ with the following properties:
        \begin{itemize}
            \item $\rho' \otimes_{\cO_1} \overline{\mathbb{Q}}_p$ is modular, associated to a cuspidal, $\iota$-ordinary regular algebraic automorphic representation $\pi$ of $\GL_2(\mathbb{A}_{F_1})$ of weight 0 and trivial central character.
            \item $\pi$ is semistable, and ramified precisely at those places $v$ of $F_1$ at which $\mathrm{WD}(\rho|_{G_{F_{1, v}}})$ is ramified.
            \item For each $v | p$ of $F_1$, the unit eigenvalues of $\Frob_v$ on $\mathrm{WD}(\rho|_{G_{F_{1, v}}})$ and $\mathrm{WD}(\rho'|_{G_{F_{1, v}}})$ have the same image in $k$.
        \end{itemize}
    \end{itemize}
    Then, if $\rho' \text{ mod } \varpi^C \cong \rho|_{G_{F_1}} \text{ mod }\varpi^C$, then $\rho|_{G_{F_1}}$ is modular. 
    
    To complete the proof, we therefore just need to exhibit a pair $(F_1, \rho')$ with the given list of properties (as we can take $F' = F_1$). The existence of a pair with the required properties is now precisely the content of Theorem \ref{thm_pot_mod}, noting that the extension $M_{\mathfrak{p}}$ constructed there may be embedded in $E_1$. 
\end{proof}

\bibliographystyle{alpha}
\bibliography{FM}

\end{document}